\documentclass[11pt,reqno]{amsart}

\usepackage{amsmath,amssymb,amsthm,mathtools}
\usepackage{bm}
\usepackage{mathrsfs}
\usepackage{microtype}
\usepackage{geometry}
\usepackage{cite}
\usepackage{xcolor}
\usepackage{hyperref}
\hypersetup{
  colorlinks=true,
  linkcolor=blue,
  citecolor=blue,
  urlcolor=blue,
  filecolor=blue,
  pdftitle={Stability of electrokinetic Couette states in the Poisson-Nernst-Planck-Navier-Stokes system},
  pdfauthor={Jiaojiao Pan and Luqi Wang}
}

\makeatletter

\def\@seccntformat#1{\csname the#1\endcsname\quad}

\def\section{\@startsection{section}{1}
  \z@{-1.45\linespacing\@plus-.45\linespacing\@minus-.15\linespacing}
  {.62\linespacing}
  {\normalfont\Large\bfseries\raggedright}}
\def\subsection{\@startsection{subsection}{2}
  \z@{-1.05\linespacing\@plus-.35\linespacing\@minus-.12\linespacing}
  {.48\linespacing}
  {\normalfont\large\bfseries\raggedright}}
\def\subsubsection{\@startsection{subsubsection}{3}
  \z@{-.8\linespacing\@plus-.3\linespacing\@minus-.1\linespacing}
  {.38\linespacing}
  {\normalfont\normalsize\bfseries\raggedright}}

\def\@settitle{
  \begin{center}
    \vspace*{.02in}
    {\LARGE\bfseries\@title\par}
  \end{center}
}

\def\@setauthors{
  \begin{center}
    \vspace{1.15em}
    {\large\authors\par}
  \end{center}
}

\renewenvironment{abstract}{
  \global\setbox\abstractbox=\vtop\bgroup
  \hsize=.88\textwidth
  \normalfont\normalsize
  \noindent\ignorespaces
}{
  \par\egroup
  \ifx\@setabstract\relax \@setabstracta \fi
}

\def\@setabstract{\@setabstracta \global\let\@setabstract\relax}
\def\@setabstracta{
  \ifvoid\abstractbox\else
    \vspace{2.05em}
    \begin{center}{\bfseries Abstract}\end{center}
    \vspace{.18em}
    \noindent\hfil\box\abstractbox\hfil\par
  \fi
}

\def\@adminfootnotes{}
\def\@maketitle{
  \normalfont\normalsize
  \@mkboth{\@nx\shortauthors}{\@nx\shorttitle}
  \@settitle
  \ifx\@empty\authors\else \@setauthors \fi
  \@setabstract
  \normalsize
  \vspace{1.05em}
  \ifx\@empty\@keywords\else
    \noindent{\bfseries Keywords.}\enspace\@keywords\par
  \fi
  \vspace{.30em}
  \ifx\@empty\@subjclass\else
    \noindent{\bfseries 2020 Mathematics Subject Classification.}\enspace\@subjclass\par
  \fi
  \vspace{1.20em}
}
\makeatother

\numberwithin{equation}{section}

\newtheorem{theorem}{Theorem}[section]
\newtheorem{proposition}[theorem]{Proposition}
\newtheorem{lemma}[theorem]{Lemma}

\theoremstyle{remark}
\newtheorem{remark}{Remark}[section]

\newcommand{\R}{\mathbb{R}}
\newcommand{\C}{\mathbb{C}}
\newcommand{\T}{\mathbb{T}}
\newcommand{\e}{\mathrm{e}}
\newcommand{\ii}{\mathrm{i}}
\newcommand{\dd}{\,\mathrm{d}}
\newcommand{\Div}{\operatorname{div}}
\newcommand{\Rey}{\mathrm{Re}}
\newcommand{\Pdiv}{\mathcal{P}}
\newcommand{\AS}{\mathcal{A}_{\mathrm S}}
\newcommand{\AD}{\mathcal{A}_{\mathrm D}}
\newcommand{\Lzero}{\mathcal{L}_0}
\newcommand{\Lzeta}{\mathcal{L}_{\zeta}}
\newcommand{\LC}{\mathcal{L}_{\mathrm C}}
\newcommand{\Li}{\mathcal{L}_{\mathrm{ion}}}
\newcommand{\Hphase}{\mathscr{H}}
\newcommand{\Dphase}{\mathscr{D}}
\newcommand{\Xtheta}{\mathscr{X}_{\theta}}
\newcommand{\Nzero}{\mathcal{N}_0}
\newcommand{\Kop}{\mathcal{K}}
\newcommand{\Nzeta}{\mathcal{N}_{\zeta}}
\newcommand{\Id}{\mathcal{I}}
\newcommand{\norm}[1]{\left\lVert #1\right\rVert}

\title[Stability of electrokinetic Couette states]{Stability of electrokinetic Couette states in the Poisson--Nernst--Planck--Navier--Stokes system}
\author[J. Pan and L. Wang]{Jiaojiao Pan\textsuperscript{*}\hspace{4.5em}Luqi Wang\textsuperscript{\(\dagger\)}}
\date{}

\subjclass[2020]{76E25, 76E05, 35B35, 35Q35, 76D05}
\keywords{Poisson--Nernst--Planck--Navier--Stokes system; Electrokinetic Couette states; Nonlinear stability; Wall potentials; Enhanced dissipation}

\begin{document}
\pagestyle{plain}
\setlength{\skip\footins}{8pt plus 2pt minus 2pt}
\renewcommand{\footnoterule}{
  \kern -3pt
  \hrule width .30\textwidth height .4pt
  \kern 2.6pt
}
\raggedbottom

\begin{abstract}
In this paper, we study a two-dimensional Poisson--Nernst--Planck--Navier--Stokes system in a periodic channel driven by wall motion and an imposed tangential electric field, with two ionic species that may have unequal diffusivities.  At the electroneutral Couette state, a diffusivity-weighted ionic energy combines with the triangular structure of the linearized operator to yield exponential linear stability for fixed \((A,E_0)\) and all positive diffusivities.  Small-data nonlinear exponential stability is then obtained on a complex interpolation space adapted to the graph domain of the generator by combining analytic-semigroup smoothing with quadratic estimates.  For unequal diffusivities, the nonzero streamwise modes of the linearized ionic subsystem satisfy a bounded-channel enhanced-dissipation estimate at rate \(D_{\min}^{1/3}|A|^{2/3}\) under an explicit strong-shear condition.  When \(D_+=D_-\), a Fourier-mode factorization separates the common Couette advection--diffusion operator from a contractive drift--reaction semigroup, so the same scalar mixing rate is retained without any smallness condition on the electrostatic coupling.  We also construct an exact Poisson--Boltzmann/electroosmotic Couette family for prescribed wall potentials.  For sufficiently small wall-potential amplitude, its linearized generator is a small graph-domain perturbation of the electroneutral generator, which yields linear and nonlinear exponential stability on the same interpolation scale.
\end{abstract}

\maketitle

\begingroup
\renewcommand{\thefootnote}{\fnsymbol{footnote}}
\footnotetext[1]{School of Mathematics, Nanjing University, Nanjing 210093, China, panjiaojiao.math@gmail.com}
\footnotetext[2]{School of Mathematics, Nanjing University, Nanjing 210093, China, wangluqi@nju.edu.cn}
\endgroup

\thispagestyle{plain}

\section{Introduction}

The Poisson--Nernst--Planck--Navier--Stokes (PNPNS) system describes ionic electrodiffusion in a viscous incompressible fluid coupled to a self-consistent electrostatic potential.  An early weak-solution framework for the coupled Navier--Stokes--Nernst--Planck--Poisson system in the presence of an externally applied electric field was established by Schmuck \cite{Schmuck2009}.  Bothe, Fischer and Saal \cite{BotheFischerSaal2014} proved local strong well-posedness in bounded domains in arbitrary space dimension and, in two dimensions, global strong well-posedness together with exponential convergence to equilibrium.  Global two-dimensional results under blocking or selective ionic boundary conditions were obtained by Constantin and Ignatova \cite{ConstantinIgnatova2019}.  Constantin, Ignatova and Lee studied three-dimensional perturbations near equilibrium in \cite{ConstantinIgnatovaLeeNear2022}, treated solutions far from equilibrium in \cite{ConstantinIgnatovaLeeFar2021}, developed a weak-current stability theory for nonequilibrium steady states in \cite{ConstantinIgnatovaLeePhysD2022}, and analyzed interior electroneutrality in the small-Debye-length regime in \cite{ConstantinIgnatovaLeeElectroneutral2021}.  Lee \cite{Lee2023} proved global regularity in three-dimensional bounded domains with mixed ionic boundary conditions and Robin data for the potential.  For the Navier--Stokes coupling, that result is conditional on velocity regularity.  On periodic domains, exponential nonlinear stability of constant steady states, together with a large-data two-dimensional result for general collections of species and diffusivities, was established by Abdo and Ignatova \cite{AbdoIgnatova2024}.  Zhai and Wu \cite{ZhaiWu2024} studied well-posedness and decay rates for the PNPNS initial-value problem, while Zhang and Zhang \cite{ZhangZhang2024} treated the quasi-neutral limit in smooth bounded domains and Shen and Wang \cite{ShenWang2022} considered a compressible PNPNS Cauchy problem with a doping profile.  For the Nernst--Planck--Stokes system with nonequilibrium Dirichlet data, Lee established long-time dynamical bounds and a finite-dimensional global attractor in \cite{Lee2024}, and later proved asymptotic stability of sufficiently weak nonequilibrium steady states in \cite{Lee2025}.  More recently, the small-convection limit from the Nernst--Planck--Navier--Stokes system to the corresponding Nernst--Planck--Stokes system under Dirichlet potential boundary data was quantified by Du, Wang and Zhang \cite{DuWangZhang2026}.

Classical electrokinetic theory provides the physical background for the wall-potential and electric-double-layer terms used below.  A standard continuum treatment of electroosmosis and related physicochemical transport mechanisms was given by Probstein \cite{Probstein1994}.  Hunter \cite{Hunter1981} developed the relation between zeta potential, interfacial charge and the electric double layer from the viewpoint of colloid science.  Rubinstein and Zaltzman \cite{RubinsteinZaltzman2000} showed that nonequilibrium electroosmosis at a permselective interface can drive convective instability.  They later derived an electro-osmotic slip description and used it in a linear electroconvective stability analysis in \cite{ZaltzmanRubinstein2007}.  These membrane-driven mechanisms differ from the reservoir-driven Couette configuration considered here, yet they illustrate how wall potential and double-layer structure can modify hydrodynamic behavior.

A central feature of the present problem is the wall-driven parallel shear in the reference velocity.  General nonequilibrium steady states that may have nonzero fluid velocity were constructed by Constantin, Ignatova and Lee \cite{ConstantinIgnatovaLeePhysD2022}, whereas their Section~4 weak-current stability theorem concerned one-dimensional steady current solutions with the fluid at rest.  The stability analysis in Lee \cite{Lee2025} concerned Stokes rather than Navier--Stokes coupling.  Neither stability result therefore covers a wall-driven Couette reference flow.

On the hydrodynamic side, Romanov \cite{Romanov1973} established the classical spectral stability of plane Couette flow, while Heck, Kim and Kozono \cite{HeckKimKozono2009} proved exponential stability under small tangentially periodic perturbations for fixed positive viscosity.  Nonlinear inviscid damping near planar Couette flow for the two-dimensional Euler equation on \(\mathbb T\times\mathbb R\) was proved by Bedrossian and Masmoudi \cite{BedrossianMasmoudi2015}.  In the viscous unbounded-cross-stream setting, Bedrossian, Masmoudi and Vicol \cite{BedrossianMasmoudiVicol2016} identified enhanced dissipation and inviscid damping near two-dimensional Couette flow, and Masmoudi and Zhao \cite{MasmoudiZhao2020} refined the enhanced-dissipation analysis in critical regularity.  Boundary effects require separate arguments in finite channels: Ionescu and Jia \cite{IonescuJia2020} proved inviscid damping and asymptotic stability for Euler, while Bedrossian and He \cite{BedrossianHe2020} obtained no-slip boundary-layer estimates for the linearized Navier--Stokes equations.  Wei, Zhang and Zhao \cite{WeiZhangZhao2018} established linear inviscid damping for monotone shear flows, and Wei and Zhang \cite{WeiZhang2021,WeiZhang2023,WeiZhang2026} developed transition-threshold and enhanced-dissipation results for viscous Couette flow, including a finite-channel Navier-slip result.  For passive scalars, Bedrossian and Coti Zelati \cite{BedrossianCotiZelati2017} quantified enhanced dissipation and hypoelliptic regularization for shear flows, while Albritton, Beekie and Novack \cite{AlbrittonBeekieNovack2022} obtained the \(\nu^{-1/3}\) enhanced-dissipation time scale for monotone shears in bounded channels with homogeneous Dirichlet data.

The paper establishes linear exponential stability of the electroneutral Couette state for unequal positive ionic diffusivities and small-data nonlinear exponential stability on a complex interpolation space adapted to the graph domain of the linearized generator.  It also gives an enhanced-dissipation estimate for the nonzero streamwise ionic modes under an explicit strong-shear condition.  When the diffusivities coincide, an exact Fourier-mode factorization removes the electrostatic-coupling smallness condition.  A Poisson--Boltzmann/electroosmotic Couette family is constructed for prescribed wall potentials and is shown to be stable for sufficiently small wall-potential amplitude.

For the electroneutral reference state, the ionic concentrations are constant and the velocity is \(Ay\bm e_1\).  The ionic block of the linearization admits a coercive diffusivity-weighted \(L^2\) identity for arbitrary \(D_+,D_->0\).  Because constant base concentrations remove the linear velocity forcing from the ionic equations, while the charge perturbation acts on the fluid equation, the full linearized operator has a triangular block structure.  Combining ionic decay with the Couette semigroup yields exponential linear stability.  The same generator then supports the nonlinear argument on the interpolation scale between the base space and the graph domain of the generator.

Quantitative mixing is obtained by writing the unequal-diffusivity ionic system as two scalar Couette advection--diffusion generators coupled by a constant Poisson reaction matrix.  The bounded-channel estimate of Albritton, Beekie and Novack \cite{AlbrittonBeekieNovack2022}, together with a Duhamel argument, gives decay at rate \(D_{\min}^{1/3}|A|^{2/3}\) when the shear dominates the zero-order ionic coupling.  The resulting sufficient condition displays explicitly the dependence on \(D_\pm\), \(\lambda_{\mathrm D}\) and \(c_0/\lambda_{\mathrm D}^2\).  In the equal-diffusivity case, a Fourier-mode factorization separates the common Couette advection--diffusion operator from a contractive drift--reaction semigroup, so the scalar mixing rate is retained without any smallness condition on the electrostatic coupling.

For prescribed wall potentials, the one-dimensional Poisson--Boltzmann profile determines an electroosmotic correction to the affine Couette velocity through the streamwise momentum balance.  When the wall-potential amplitude is small, the corresponding linearized operator is a small graph-domain perturbation of the electroneutral generator.  A resolvent perturbation argument preserves exponential stability, and the nonlinear stability proof proceeds on the same interpolation scale.

The rest of the paper is organized as follows.  Section~\ref{sec:main} formulates the PNPNS model, fixes the functional setting and states the principal results.  Section~\ref{sec:linear-proof} constructs the exact electroosmotic Couette state and establishes the linear estimates, including the graph-norm equivalence used later.  Section~\ref{sec:nonlinear-proof} proves nonlinear stability by combining the quadratic mapping property with analytic-semigroup smoothing in a weighted fixed-point argument.  Section~\ref{sec:ED-proof} develops the coupled ionic enhanced-dissipation estimate and the equal-diffusivity factorization, and Section~\ref{sec:PB-proof} treats weak Poisson--Boltzmann layers.  Appendix~\ref{app:perturbation} contains the sectorial perturbation lemma used in the weak-layer stability argument of Section~\ref{sec:PB-proof}.

\section{Preliminaries and main results}\label{sec:main}

\subsection{Model and exact steady states}

Write \(\T_L:=\R/(L\mathbb Z)\) for the one-dimensional torus of period \(L>0\).  After nondimensionalization by the channel half-gap, let
\[
 \Omega:=\T_{2\pi}\times(-1,1),
 \qquad (x,y)\in\Omega.
\]
We write \(\bm u=(u_1,u_2)\) for the velocity, \(p\) for the pressure, \(c_+\) and \(c_-\) for the ionic concentrations and \(\phi\) for the self-consistent potential.  Each field is periodic in \(x\).  With \(\bm e_1=(1,0)\), the total electric field is \(E_0\bm e_1-\nabla\phi\), where the constant vector \(E_0\bm e_1\) represents an imposed tangential electromotive forcing.  Hence the affine total potential \(-E_0x+\phi\) itself need not be single-valued on the torus.  We take the pressure to be periodic and impose no independent mean pressure gradient.  The governing equations are
\begin{equation}\label{eq:PNPNS}
\left\{
\begin{aligned}
 \partial_t\bm u+(\bm u\cdot\nabla)\bm u+\nabla p
 &=\nu\Delta\bm u+\gamma \varrho\bigl(E_0\bm e_1-\nabla\phi\bigr),\\
 \Div\bm u&=0,\\
 \partial_t c_+ +\bm u\cdot\nabla c_+
 &=D_+\Div\!\left(\nabla c_+ +c_+\nabla\phi-E_0c_+\bm e_1\right),\\
 \partial_t c_- +\bm u\cdot\nabla c_-
 &=D_-\Div\!\left(\nabla c_- -c_-\nabla\phi+E_0c_-\bm e_1\right),\\
 -\lambda_{\mathrm D}^2\Delta\phi&=\varrho:=c_+-c_-.
\end{aligned}
\right.
\end{equation}
Here \(\nu,D_+,D_-,\lambda_{\mathrm D},\gamma>0\), \(E_0\in\R\), and \(\varrho=c_+-c_-\) is the signed ionic charge density in the chosen nondimensionalization.  The boundary data are
\begin{equation}\label{eq:boundary}
\left\{
\begin{aligned}
 \bm u(x,\pm1,t)&=(U_\pm,0),\\
 \phi(x,\pm1,t)&=\zeta_\pm,\\
 c_+(x,\pm1,t)&=c_0\e^{-\zeta_\pm},\\
 c_-(x,\pm1,t)&=c_0\e^{\zeta_\pm},
\end{aligned}
\right.
\end{equation}
where \(c_0>0\). 
At the plates, these reservoir values impose \(\log(c_\pm/c_0)\pm\phi=0\) and therefore yield homogeneous Dirichlet conditions for perturbations.  A nonzero tangential field \(E_0\) can still drive a tangential current through the resulting equilibrium profile.

Set
\[
 U_{\mathrm C}(y):=Ay+B,
 \qquad
 A:=\frac{U_+-U_-}{2},
 \qquad
 B:=\frac{U_++U_-}{2},
\]
and
\[
 \ell_\zeta(y):=\frac{1+y}{2}\zeta_+
 +\frac{1-y}{2}\zeta_-,
 \qquad
 \delta_\zeta:=|\zeta_+|+|\zeta_-|.
\]
For these boundary data, the associated Poisson--Boltzmann problem is
\begin{equation}\label{eq:PB}
\left\{
\begin{aligned}
 -\lambda_{\mathrm D}^2\phi_*''
 &=c_0\e^{-\phi_*}-c_0\e^{\phi_*},
 &&-1<y<1,\\
 \phi_*(\pm1)&=\zeta_\pm.
\end{aligned}
\right.
\end{equation}

We first record the exact reference family used throughout the analysis.

\begin{proposition}[Exact electroosmotic Couette state]\label{prop:exact}
For every \(\lambda_{\mathrm D},c_0>0\) and \(\zeta_\pm\in\R\), problem \eqref{eq:PB} has a unique smooth solution.  With
\[
 \chi:=\frac{\gamma\lambda_{\mathrm D}^2}{\nu},
\]
the system \eqref{eq:PNPNS}--\eqref{eq:boundary} admits the steady state
\begin{equation}\label{eq:base-state}
\left\{
\begin{aligned}
 c_{+,*}(y)&=c_0\e^{-\phi_*(y)},\\
 c_{-,*}(y)&=c_0\e^{\phi_*(y)},\\
 \varrho_*(y)&=c_{+,*}(y)-c_{-,*}(y),\\
 \bm u_*(y)&=U_*(y)\bm e_1,\\
 U_*(y)&=U_{\mathrm C}(y)
 +\chi E_0\bigl(\phi_*(y)-\ell_\zeta(y)\bigr),\\
 p_*(y)&=\gamma\bigl(c_{+,*}(y)+c_{-,*}(y)\bigr)+p_0,
 \qquad p_0\in\R.
\end{aligned}
\right.
\end{equation}
\end{proposition}

Formula \eqref{eq:base-state} for the electroosmotic correction follows directly from the Poisson equation and requires no linearization of \eqref{eq:PB}.

\subsection{Functional setting}

For \(\zeta_+=\zeta_-=0\), Proposition~\ref{prop:exact} gives the uniform state
\begin{equation}\label{eq:uniform-base}
 \phi_*=0,
 \qquad c_{+,*}=c_{-,*}=c_0,
 \qquad \varrho_*=0,
 \qquad \bm u_*=U_{\mathrm C}(y)\bm e_1.
\end{equation}
To remove the constant mean wall speed \(B\), set
\[
 \widetilde x:=x-Bt,\qquad \widetilde y:=y,\qquad \widetilde t:=t,
\]
and, for each scalar field \(f\in\{p,c_+,c_-,\phi\}\), define
\[
 \widetilde f(\widetilde x,y,t):=f(\widetilde x+Bt,y,t),
 \qquad
 \widetilde{\bm u}(\widetilde x,y,t)
 :=\bm u(\widetilde x+Bt,y,t)-B\bm e_1.
\]
Under this change, \(U_{\mathrm C}(y)=Ay+B\) becomes \(Ay\), while the perturbation equations retain the same form.  Dropping tildes from now on, we set
\[
 \bm U(y):=Ay\bm e_1,
 \qquad
 \Rey:=\frac{|A|}{\nu},
 \qquad
 D_{\min}:=\min\{D_+,D_-\},
 \qquad
 D_{\max}:=\max\{D_+,D_-\}.
\]

All Sobolev spaces below are periodic in \(x\).  Here \(H_0^1(\Omega)\) denotes the subspace with zero trace on the plates \(y=\pm1\), together with periodicity in \(x\).  Let \(L^2_\sigma(\Omega)\) be the closed subspace of \(L^2(\Omega;\R^2)\) obtained as the \(L^2\)-closure of smooth \(2\pi\)-periodic divergence-free vector fields whose normal component vanishes on the plates, and let \(\Pdiv\) denote the corresponding Leray projection.  Define the Dirichlet Stokes operator \(\AS\) and the positive Dirichlet Laplacian \(\AD\) by
\begin{align*}
 \AS\bm v&:=-\Pdiv\Delta\bm v,
\qquad \mathcal D(\AS)=H^2(\Omega;\R^2)\cap H_0^1(\Omega;\R^2)\cap L^2_\sigma(\Omega),\\
 \AD f&:=-\Delta f,
\qquad \mathcal D(\AD)=H^2(\Omega)\cap H_0^1(\Omega).
\end{align*}
We shall also use the Dirichlet Poisson solution operator
\[
 \Kop:=\lambda_{\mathrm D}^{-2}\AD^{-1},
 \qquad
 -\lambda_{\mathrm D}^2\Delta(\Kop f)=f,
 \qquad (\Kop f)|_{y=\pm1}=0.
\]
Set the base and graph spaces as
\[
 \Hphase:=L^2_\sigma(\Omega)\times L^2(\Omega)\times L^2(\Omega),
 \qquad
 \Dphase:=\mathcal D(\AS)\times\mathcal D(\AD)^2.
\]
On \(\Hphase\) we use the Hilbert product norm
\[
 \|(\bm v,f,g)\|_{\Hphase}^2
 :=\|\bm v\|_{L^2}^2+\|f\|_{L^2}^2+\|g\|_{L^2}^2.
\]
We use the same Euclidean product convention on \(L^2(\Omega)^2\).
We endow \(\Dphase\) with the norm
\[
 \|(\bm v,f,g)\|_{\Dphase}
 :=\|\bm v\|_{H^2}+\|f\|_{H^2}+\|g\|_{H^2},
\]
which is equivalent to the product of the Stokes and Dirichlet-Laplacian graph norms.
For spectral and analytic-semigroup arguments, we use the complexifications of these spaces without changing the notation.  The linear semigroups and nonlinear maps preserve the real subspaces, and the nonlinear stability statements below concern real-valued initial data.  Let \([X_0,X_1]_\theta\) denote the complex interpolation space of a compatible Banach couple \((X_0,X_1)\).  We write evolution equations in the form \(\partial_t Z=\mathcal L Z\).  
We call \(\mathcal L\) sectorial if \(-\mathcal L+\omega\Id\) is sectorial of angle less than \(\pi/2\) for some real shift \(\omega\).  For \(\theta\in(3/4,1)\), set
\[
 \Xtheta:=[\Hphase,\Dphase]_\theta.
\]
To record the Sobolev regularity inherited from this interpolation space, let
\[
 \mathscr E_0:=L^2(\Omega;\C^2)\times L^2(\Omega)^2,
 \qquad
 \mathscr E_1:=H^2(\Omega;\C^2)\times H^2(\Omega)^2.
\]
Both embeddings \(\Hphase\hookrightarrow\mathscr E_0\) and \(\Dphase\hookrightarrow\mathscr E_1\) are continuous.  Functoriality of complex interpolation, together with the standard identity \([L^2,H^2]_\theta=H^{2\theta}\) on the periodic strip, yields
\begin{equation}\label{eq:Xembed}
 \Xtheta\hookrightarrow
 H^{2\theta}(\Omega;\C^2)\times H^{2\theta}(\Omega)^2
 \hookrightarrow C^0(\overline\Omega;\C^2)\times C^0(\overline\Omega)^2.
\end{equation}
For the interpolation facts, see \cite[Chapters~2 and~4]{BerghLofstrom1976}.  Since \(\Dphase\) is dense in \(\Hphase\), it is also dense in \(\Xtheta\).  Moreover, \(2\theta>1/2\), so the Sobolev trace map is continuous on \(H^{2\theta}(\Omega)\).  Hence the zero plate traces of elements of \(\Dphase\) pass to \(\Xtheta\) by density.  The velocity component remains divergence free because \(L^2_\sigma(\Omega)\) is closed in \(L^2(\Omega;\C^2)\).  Thus \(\Xtheta\) already incorporates the homogeneous perturbation traces and the divergence constraint.

\begin{remark}[Choice of the interpolation exponent]
We choose \(\theta>3/4\), hence \(2\theta>3/2\), so that the embedding in \eqref{eq:Xembed} and the elementary product estimates used below control
\[
 \bm v\cdot\nabla\bm v,
 \qquad
 \bm v\cdot\nabla a_\pm,
 \qquad
 \Div(a_\pm\nabla\Kop(a_+-a_-))
\]
in \(L^2(\Omega)\) without endpoint trace or multiplication arguments.  The same choice is used for the weak Poisson--Boltzmann problem so that both nonlinear stability theorems are posed on one interpolation scale.
\end{remark}

Unless a subscript indicates otherwise, \(C>0\) denotes a generic constant that may change from line to line and may depend on the fixed physical parameters, the channel, and the fixed interpolation exponents used in that estimate.  Constants such as \(C_{\mathrm{ED}}\), \(d_{\mathrm{ED}}\), and \(C_{\mathrm{PB}}\) are kept fixed within the estimates in which they are introduced.  We write \(C_{\mathrm{emb}}\) for a fixed norm of the embedding
\[
 \Xtheta\hookrightarrow
 L^\infty(\Omega;\C^2)\times L^\infty(\Omega)^2
\]
that follows from \eqref{eq:Xembed}.

For a scalar or vector field \(f\), define its streamwise mean and nonzero-mode part componentwise by
\[
 (\mathscr P_0f)(y):=\frac{1}{2\pi}\int_0^{2\pi}f(x,y)\dd x,
 \qquad
 \mathscr P_{\neq 0}:=\Id-\mathscr P_0,
 \qquad f_{\neq 0}:=\mathscr P_{\neq 0}f.
\]
Both projections are orthogonal in \(L^2\) and preserve the Dirichlet trace whenever the trace is defined.

\subsection{Main results}

We begin with fixed-parameter stability of the electroneutral state \eqref{eq:uniform-base}.

\begin{theorem}[Linear exponential stability]\label{thm:linear-main}
Fix
\[
 A,E_0\in\R,
 \qquad
 \nu,D_+,D_-,\lambda_{\mathrm D},\gamma,c_0>0.
\]
The linearization of \eqref{eq:PNPNS} about \eqref{eq:uniform-base} defines a sectorial operator \(\Lzero\) on \(\Hphase\) with domain \(\Dphase\).  There exist \(M_0\ge1\) and \(\omega_0>0\), depending on the fixed parameters, such that
\begin{equation}\label{eq:main-linear-decay}
 \norm{\e^{t\Lzero}}_{\mathcal L(\Hphase)}
 \le M_0\e^{-\omega_0t},
 \qquad t\ge0.
\end{equation}
Hence the uniform electrokinetic Couette state is linearly exponentially stable for every fixed finite \(\Rey\) and for arbitrary positive ionic diffusivities.
\end{theorem}

On the same interpolation scale, the linear generator also controls the nonlinear perturbation problem.

\begin{theorem}[Nonlinear exponential stability]\label{thm:nonlinear-main}
Let \(\theta\in(3/4,1)\) and fix the parameters as in Theorem~\ref{thm:linear-main}.  There exist \(\varepsilon_*>0\), \(C\ge1\), and \(\omega>0\) such that every real-valued initial perturbation
\[
 \bm Z_{\mathrm{in}}
 =\bigl(\bm v_{\mathrm{in}},a_{+,\mathrm{in}},a_{-,\mathrm{in}}\bigr)
 \in\Xtheta,
 \qquad
 \norm{\bm Z_{\mathrm{in}}}_{\Xtheta}\le\varepsilon_*,
\]
generates a unique global mild solution in the class
\[
 \bm Z\in C([0,\infty);\Xtheta)
 \cap C((0,\infty);\Dphase)
 \cap C^1((0,\infty);\Hphase).
\]
Moreover,
\begin{equation}\label{eq:main-nonlinear-decay}
 \norm{\bm Z(t)}_{\Xtheta}
 \le C\e^{-\omega t}
 \norm{\bm Z_{\mathrm{in}}}_{\Xtheta},
 \qquad t\ge0.
\end{equation}
After decreasing \(\varepsilon_*\) if necessary, the corresponding concentrations satisfy
\[
 c_+(x,y,t)\ge \frac{c_0}{2},
 \qquad
 c_-(x,y,t)\ge \frac{c_0}{2},
 \qquad (x,y)\in\Omega,\ t\ge0.
\]
\end{theorem}

The next result separates two regimes.  For unequal diffusivities, a Duhamel argument yields enhanced dissipation when Couette mixing dominates the zero-order electrostatic coupling.  For equal diffusivities, an exact factorization removes this coupling condition and recovers the full scalar Couette rate.

\begin{theorem}[Ionic enhanced dissipation with an equal-diffusivity refinement]\label{thm:ED-main}
Assume \(A\ne0\), and set
\[
 \kappa_{\mathrm D}:=\frac{c_0}{\lambda_{\mathrm D}^2}.
\]
There exist constants \(\varepsilon_{\mathrm{ED}}>0\), \(C_{\mathrm{ED}}\ge1\), and \(d_{\mathrm{ED}}>0\), depending only on the normalized Couette profile and the channel geometry, such that if
\begin{equation}\label{eq:ED-strong-shear}
 \frac{D_{\max}}{|A|}\le\varepsilon_{\mathrm{ED}},
 \qquad
 4C_{\mathrm{ED}}\kappa_{\mathrm D}D_{\max}
 \le d_{\mathrm{ED}}D_{\min}^{1/3}|A|^{2/3},
\end{equation}
then every solution \((a_+,a_-)\) of the linearized ionic subsystem on
\(\Omega=\T_{2\pi}\times(-1,1)\) with homogeneous Dirichlet conditions satisfies
\begin{equation}\label{eq:main-ED}
 \norm{(a_{+,\neq 0},a_{-,\neq 0})(t)}_{L^2\times L^2}
 \le C_{\mathrm{ED}}
 \exp\!\left[-\frac{d_{\mathrm{ED}}}{2}
 D_{\min}^{1/3}|A|^{2/3}t\right]
 \norm{(a_{+,\neq 0},a_{-,\neq 0})(0)}_{L^2\times L^2},
 \qquad t\ge0.
\end{equation}
The constants in \eqref{eq:main-ED} are independent of \(E_0\).

If \(D_+=D_-=D\), the second condition in \eqref{eq:ED-strong-shear} is unnecessary.  Under the single assumption
\[
 \frac{D}{|A|}\le\varepsilon_{\mathrm{ED}},
\]
the sharper estimate
\begin{equation}\label{eq:main-ED-equal}
 \norm{(a_{+,\neq 0},a_{-,\neq 0})(t)}_{L^2\times L^2}
 \le C_{\mathrm{ED}}
 \exp\!\left[-d_{\mathrm{ED}}D^{1/3}|A|^{2/3}t\right]
 \norm{(a_{+,\neq 0},a_{-,\neq 0})(0)}_{L^2\times L^2}
\end{equation}
holds for all \(t\ge0\), uniformly in \(E_0\) and \(\kappa_{\mathrm D}\).
\end{theorem}

\begin{remark}[Unequal and equal diffusivity mechanisms]
For \(D_+\ne D_-\), the two species evolve under different scalar advection--diffusion generators.  The proof of \eqref{eq:main-ED} therefore treats the Poisson reaction as a bounded coupling, and the second condition in \eqref{eq:ED-strong-shear} is a sufficient condition for retaining half of the slower scalar mixing rate.  In particular, \eqref{eq:main-ED} reduces the nonzero-mode norm by at least a factor of two whenever
\[
 t\ge \frac{2\log(2C_{\mathrm{ED}})}{d_{\mathrm{ED}}}
 D_{\min}^{-1/3}|A|^{-2/3}.
\]
When \(D_+=D_-=D\), the two species share a common scalar Couette generator and the remaining drift--reaction semigroup is contractive and commuting.  This yields \eqref{eq:main-ED-equal} without an electrostatic-coupling restriction, with a factor-two relaxation bound
\[
 t\ge \frac{\log(2C_{\mathrm{ED}})}{d_{\mathrm{ED}}}
 D^{-1/3}|A|^{-2/3}.
\]
The operator obstruction to extending this exact factorization to unequal diffusivities is described in Remark~\ref{rem:ED-obstruction}.
\end{remark}

The last main result extends the stability theory from the electroneutral reference state to weak Poisson--Boltzmann layers with arbitrary positive ionic diffusivities.

\begin{theorem}[Weak Poisson--Boltzmann layers]\label{thm:PB-main}
Let \(\theta\in(3/4,1)\) and fix
\[
 A,E_0\in\R,
 \qquad
 \nu,D_+,D_-,\lambda_{\mathrm D},\gamma,c_0>0.
\]
There exists
\[
 \delta_*=
 \delta_*(A,E_0,\nu,D_+,D_-,\lambda_{\mathrm D},\gamma,c_0)>0
\]
such that, if \(\delta_\zeta<\delta_*\), then the exact state \eqref{eq:base-state} satisfies the following assertions.
\begin{enumerate}
\item Its linearized generator \(\Lzeta\), acting on \(\Hphase\) with domain \(\Dphase\), generates an exponentially stable analytic semigroup.  More precisely, there exist \(M_\zeta\ge1\) and \(\omega_\zeta>0\) such that
\begin{equation}\label{eq:PB-main-linear}
 \norm{\e^{t\Lzeta}}_{\mathcal L(\Hphase)}
 \le M_\zeta\e^{-\omega_\zeta t},
 \qquad t\ge0.
\end{equation}
\item There exist \(\varepsilon_\zeta>0\), \(C_{\zeta,\mathrm{nl}}\ge1\), and \(\omega_{\zeta,\mathrm{nl}}>0\) such that every real-valued perturbation \(\bm Y_{\mathrm{in}}\in\Xtheta\) with
\[
 \norm{\bm Y_{\mathrm{in}}}_{\Xtheta}\le\varepsilon_\zeta
\]
generates a unique global mild solution in the class
\[
 \bm Y\in C([0,\infty);\Xtheta)
 \cap C((0,\infty);\Dphase)
 \cap C^1((0,\infty);\Hphase),
\]
and
\begin{equation}\label{eq:PB-main-nonlinear}
 \norm{\bm Y(t)}_{\Xtheta}
 \le C_{\zeta,\mathrm{nl}}\e^{-\omega_{\zeta,\mathrm{nl}}t}
 \norm{\bm Y_{\mathrm{in}}}_{\Xtheta},
 \qquad t\ge0.
\end{equation}
Writing \(\bm Y=(\bm v,b_+,b_-)\) in species coordinates and setting
\[
 m_*:=\min_{y\in[-1,1]}\min\{c_{+,*}(y),c_{-,*}(y)\}>0,
\]
the threshold \(\varepsilon_\zeta\) can be chosen so that the perturbed concentrations satisfy
\[
 c_{\pm,*}(y)+b_\pm(x,y,t)\ge \frac{m_*}{2}
 \qquad (x,y)\in\Omega,\ t\ge0.
\]
\end{enumerate}
\end{theorem}

\section{The exact state and linear stability}\label{sec:linear-proof}

We first prove Proposition~\ref{prop:exact}.  The resulting identities also determine the coefficients in the nonuniform linearization used in Section~\ref{sec:PB-proof}.

\begin{proof}[Proof of Proposition~\ref{prop:exact}]
Write \(\phi=\ell_\zeta+w\) with \(w\in H_0^1(-1,1)\), and consider
\[
 \mathscr F(\phi)
 :=\int_{-1}^{1}
 \left[
 \frac{\lambda_{\mathrm D}^2}{2}|\phi'|^2
 +c_0\e^{-\phi}+c_0\e^{\phi}
 \right]\dd y
\]
on the affine space
\[
 \mathscr A_\zeta
 :=\{\phi\in H^1(-1,1):\phi(\pm1)=\zeta_\pm\}.
\]
Because \(w\in H_0^1(-1,1)\), Poincar\'e's inequality gives
\[
 \|w\|_{H^1}\le C\|w'\|_{L^2}.
\]
The elementary estimate
\[
 \|\ell_\zeta'+w'\|_{L^2}^2
 \ge \frac12\|w'\|_{L^2}^2-\|\ell_\zeta'\|_{L^2}^2
\]
and the positivity of the exponential terms imply
\[
 \mathscr F(\ell_\zeta+w)
 \ge \frac{\lambda_{\mathrm D}^2}{4}\|w'\|_{L^2}^2
 -\frac{\lambda_{\mathrm D}^2}{2}\|\ell_\zeta'\|_{L^2}^2.
\]
Thus every minimizing sequence is bounded in \(H^1(-1,1)\).  Since 
\(H^1(-1,1)\hookrightarrow\hookrightarrow C^0([-1,1])\), up to a subsequence,
\[
 w_n\rightharpoonup w\quad\hbox{in }H_0^1(-1,1),
 \qquad
 w_n\to w\quad\hbox{uniformly on }[-1,1].
\]
Weak lower semicontinuity of the Dirichlet integral, together with uniform convergence of the two exponential terms, shows that \(\mathscr F\) attains its minimum at some \(\phi_*\in\mathscr A_\zeta\).

For \(h\in H_0^1(-1,1)\), the Euler--Lagrange equation is
\[
 \lambda_{\mathrm D}^2\int_{-1}^{1}\phi_*'h'\dd y
 +c_0\int_{-1}^{1}(\e^{\phi_*}-\e^{-\phi_*})h\dd y=0,
\]
which is the weak form of \eqref{eq:PB}.  Moreover, the second variation satisfies
\[
 D^2\mathscr F(\phi)[h,h]
 =\lambda_{\mathrm D}^2\|h'\|_{L^2}^2
 +c_0\int_{-1}^{1}(\e^{-\phi}+\e^{\phi})h^2\dd y>0
\]
for nonzero \(h\in H_0^1(-1,1)\).  Thus \(\mathscr F\) is strictly convex on the affine space \(\mathscr A_\zeta\).  Every weak solution of \eqref{eq:PB} is a critical point of \(\mathscr F\).  Strict convexity therefore gives uniqueness of the weak solution.  Since the right-hand side of \eqref{eq:PB} is a smooth function of \(\phi_*\), one-dimensional elliptic bootstrapping gives \(\phi_*\in C^\infty([-1,1])\).

For the steady concentrations in \eqref{eq:base-state}, differentiation gives
\[
 c_{+,*}'+c_{+,*}\phi_*'=0,
 \qquad
 c_{-,*}'-c_{-,*}\phi_*'=0.
\]
Define the nonadvective ionic fluxes by
\[
 \bm J_\pm:=-D_\pm\bigl[\nabla c_\pm
 \pm c_\pm(\nabla\phi-E_0\bm e_1)\bigr].
\]
The two identities above give \(J_{\pm,2,*}=0\), whereas
\[
 J_{+,1,*}=D_+E_0c_{+,*},
 \qquad J_{-,1,*}=-D_-E_0c_{-,*}.
\]
Because the streamwise fluxes depend only on \(y\), their divergences vanish.  Also \(\bm u_*\cdot\nabla c_{\pm,*}=0\).  Both steady Nernst--Planck equations therefore hold.

From the Poisson equation,
\(
 \varrho_*=-\lambda_{\mathrm D}^2\phi_*''
\).
Consequently, the streamwise momentum equation becomes
\[
 \nu U_*''+\gamma E_0\varrho_*=0,
\]
and hence
\[
 U_*''=\frac{\gamma\lambda_{\mathrm D}^2}{\nu}E_0\phi_*''.
\]
Integrating twice yields the formula for \(U_*\) in \eqref{eq:base-state}.  Because
\[
 \ell_\zeta(\pm1)=\zeta_\pm=\phi_*(\pm1),
\]
the electroosmotic correction vanishes at the plates and therefore \(U_*(\pm1)=U_{\mathrm C}(\pm1)=U_\pm\).  Moreover, \(\bm u_*=U_*(y)\bm e_1\) is divergence free and has zero normal component.  Finally,
\[
 \frac{\dd}{\dd y}(c_{+,*}+c_{-,*})
 =-\varrho_*\phi_*'.
\]
In the normal direction, the momentum equation reads
\[
 p_*'=-\gamma\varrho_*\phi_*'
 =\gamma\frac{\dd}{\dd y}(c_{+,*}+c_{-,*}),
\]
which gives the stated pressure up to the additive constant \(p_0\).  Together with the flux identities above, this verifies all equations and boundary conditions in \eqref{eq:PNPNS}--\eqref{eq:boundary}.
\end{proof}

Linearization about the uniform state is obtained by writing
\[
 \bm u=\bm U+\bm v,
 \qquad
 c_\pm=c_0+a_\pm,
 \qquad
 \phi=\varphi,
 \qquad
 p=p_*+\pi,
\]
and setting
\[
 q:=a_+-a_-.
\]
This gives the linearized system
\begin{equation}\label{eq:linearized-uniform}
\left\{
\begin{aligned}
 \partial_t\bm v+Ay\partial_x\bm v+Av_2\bm e_1+\nabla\pi
 &=\nu\Delta\bm v+\gamma E_0q\bm e_1,\\
 \Div\bm v&=0,\\
 \partial_t a_+ +Ay\partial_xa_+
 &=D_+\Delta a_+ +D_+c_0\Delta\varphi-D_+E_0\partial_xa_+,\\
 \partial_t a_- +Ay\partial_xa_-
 &=D_-\Delta a_- -D_-c_0\Delta\varphi+D_-E_0\partial_xa_-,\\
 -\lambda_{\mathrm D}^2\Delta\varphi&=q,
 \qquad \varphi|_{y=\pm1}=0.
\end{aligned}
\right.
\end{equation}
All perturbations \(\bm v,a_+,a_-\) satisfy homogeneous Dirichlet conditions at \(y=\pm1\).

For unequal diffusivities, the ionic block is dissipative in the following weighted energy.

\begin{lemma}[Dissipation of the ionic block]\label{lem:ionic}
Let \((a_+,a_-)\) solve the ionic part of \eqref{eq:linearized-uniform}, and define
\[
 \mathcal E_{\mathrm{ion}}(t)
 :=\frac{1}{2D_+}\norm{a_+(t)}_{L^2}^2
 +\frac{1}{2D_-}\norm{a_-(t)}_{L^2}^2.
\]
Then
\begin{equation}\label{eq:ionic-energy}
 \frac{\dd}{\dd t}\mathcal E_{\mathrm{ion}}
 +\norm{\nabla a_+}_{L^2}^2
 +\norm{\nabla a_-}_{L^2}^2
 +\frac{c_0}{\lambda_{\mathrm D}^2}\norm{a_+-a_-}_{L^2}^2
 =0.
\end{equation}
Define \(\lambda_1:=\pi^2/4\), which is the first eigenvalue of \(\AD\) on \(\Omega=\T_{2\pi}\times(-1,1)\).  Then
\begin{equation}\label{eq:ionic-exp}
 \norm{(a_+,a_-)(t)}_{L^2\times L^2}
 \le \left(\frac{D_{\max}}{D_{\min}}\right)^{1/2}
 \e^{-D_{\min}\lambda_1t}
 \norm{(a_+,a_-)(0)}_{L^2\times L^2}.
\end{equation}
\end{lemma}

\begin{proof}
Multiply the positive and negative ionic equations by \(a_+/D_+\) and \(a_-/D_-\), respectively, integrate over \(\Omega\), and add the results.  The Couette terms vanish by periodicity in \(x\).  The tangential electric drifts also vanish separately:
\[
 \int_\Omega a_\pm\partial_xa_\pm\dd x\dd y=0.
\]
Diffusion contributes
\(
 -\|\nabla a_+\|_{L^2}^2-\|\nabla a_-\|_{L^2}^2
\).
For the Poisson terms,
\[
 c_0\int_\Omega(a_+-a_-)\Delta\varphi\dd x\dd y
 =-\frac{c_0}{\lambda_{\mathrm D}^2}
 \norm{a_+-a_-}_{L^2}^2.
\]
Hence \eqref{eq:ionic-energy} holds for real-valued solutions.  On the complexified phase space, apply the same real-coefficient identity separately to the real and imaginary parts of \(a_\pm\).  Adding the two identities gives \eqref{eq:ionic-energy} with the usual complex \(L^2\)-norms.  Poincar\'e's inequality and
\[
 \norm{a_+}_{L^2}^2+\norm{a_-}_{L^2}^2
 \ge 2D_{\min}\mathcal E_{\mathrm{ion}}
\]
give
\[
 \frac{\dd}{\dd t}\mathcal E_{\mathrm{ion}}
 +2D_{\min}\lambda_1\mathcal E_{\mathrm{ion}}\le0.
\]
Since
\[
 \frac{1}{2D_{\max}}\|(a_+,a_-)\|_{L^2\times L^2}^2
 \le \mathcal E_{\mathrm{ion}}
 \le \frac{1}{2D_{\min}}\|(a_+,a_-)\|_{L^2\times L^2}^2,
\]
the preceding inequality gives \eqref{eq:ionic-exp}.
\end{proof}

Let \(\bm a=(a_+,a_-)^{\mathsf T}\).  Eliminating \(\varphi\) with the Poisson equation, the ionic generator can be written as
\[
 \Li\bm a=
 \begin{pmatrix}
 -D_+\AD a_+-Ay\partial_xa_+-D_+E_0\partial_xa_+
 -\dfrac{D_+c_0}{\lambda_{\mathrm D}^2}(a_+-a_-)\\[2mm]
 -D_-\AD a_--Ay\partial_xa_-+D_-E_0\partial_xa_-
 +\dfrac{D_-c_0}{\lambda_{\mathrm D}^2}(a_+-a_-)
 \end{pmatrix},
\]
with domain \(\mathcal D(\Li)=\mathcal D(\AD)^2\).  For relative-bound parameter
\(\varepsilon_{\mathrm{rel}}>0\), the \(H^1\)--\(H^2\) interpolation inequality together with the Dirichlet elliptic estimate gives
\[
 \|\nabla f\|_{L^2}
 \le\varepsilon_{\mathrm{rel}}\|\AD f\|_{L^2}
 +C_{\varepsilon_{\mathrm{rel}}}\|f\|_{L^2},
 \qquad f\in\mathcal D(\AD),
\]
where \(C_{\varepsilon_{\mathrm{rel}}}\) depends on
\(\varepsilon_{\mathrm{rel}}\) and the fixed channel but not on \(f\).
As \(|y|\le1\), the first-order transports \(Ay\partial_x\) and \(D_\pm E_0\partial_x\) have relative bound zero with respect to \(\operatorname{diag}(D_+\AD,D_-\AD)\), while the Poisson reaction matrix is a bounded zero-order operator on \(L^2(\Omega)^2\).  The relatively bounded perturbation theorem for sectorial operators \cite[Chapter~3, Section~3.2]{Pazy1983} shows that \(\Li\) is sectorial on \(\mathcal D(\AD)^2\).   Identity \eqref{eq:ionic-energy}, initially for domain data and then extended by density, gives the exponentially stable ionic semigroup.

For the hydrodynamic block define
\[
 \LC\bm v
 :=-\nu\AS\bm v
 -\Pdiv\bigl(Ay\partial_x\bm v+Av_2\bm e_1\bigr),
 \qquad
 \mathcal D(\LC)=\mathcal D(\AS).
\]
For the velocity block, we use the periodic-channel Couette semigroup estimate in the following lemma.

\begin{lemma}[Classical Couette semigroup estimate]\label{lem:couette-semigroup}
For fixed \(A\in\R\) and \(\nu>0\), the operator \(\LC\) generates an analytic semigroup on \(L^2_\sigma(\Omega)\).  There exist constants \(M_{\mathrm C}\ge1\) and \(\omega_{\mathrm C}>0\), depending on \((A,\nu)\), such that
\begin{equation}\label{eq:couette-semigroup}
 \norm{\e^{t\LC}}_{\mathcal L(L^2_\sigma)}
 \le M_{\mathrm C}\e^{-\omega_{\mathrm C}t},
 \qquad t\ge0.
\end{equation}
\end{lemma}

\begin{proof}
The first-order transport \(Ay\partial_x\) is \(\AS\)-bounded with relative bound zero, while \(Av_2\bm e_1\) is a bounded zero-order term.  Accordingly, the same sectorial perturbation theorem \cite[Chapter~3, Section~3.2]{Pazy1983} gives sectoriality of \(\LC\) on \(L^2_\sigma(\Omega)\) with domain \(\mathcal D(\AS)\).  If \(A=0\), the assertion follows from the spectral gap of the Dirichlet Stokes operator.

Assume \(A\ne0\) and define the rescaled Couette time \(\tau_{\mathrm C}:=|A|t\).  If \(A<0\), let \(\bm R:=\operatorname{diag}(-1,1)\) and define the unitary reflection
\[
 (\mathcal U\bm v)(x,y):=\bm R\,\bm v(-x,y).
\]
Then the reflected and rescaled problem is the normalized plane Couette linearization with viscosity
\[
 \mu:=\frac{\nu}{|A|}>0
\]
on the same periodic layer.  We use the linear estimates proved in Section~2 of Heck, Kim and Kozono \cite[Lemmas~2.3--2.4]{HeckKimKozono2009}.  In their notation, Lemma~2.3 identifies the spectral gap of the perturbed Stokes operator and the subsequent semigroup estimates are collected in Lemma~2.4, which is the linear input for their Theorem~1.1.  Their layer has plates at \(y=\pm1\), and the tangential cell \([-l,l]^{n-1}\) agrees with the present geometry for \(n=2\) and \(l=\pi\).  Taking \(q=r=2\), denote the normalized no-slip velocity semigroup by \(\mathcal S_\mu\).  Then there are constants \(M(\mu)\ge1\) and \(\delta(\mu)>0\) such that
\[
 \|\mathcal S_\mu(\tau_{\mathrm C})\|_{\mathcal L(L^2_\sigma)}
 \le M(\mu)\e^{-\delta(\mu)\tau_{\mathrm C}}.
\]
Scaling back gives \eqref{eq:couette-semigroup} with
\[
 M_{\mathrm C}=M(\nu/|A|),
 \qquad
 \omega_{\mathrm C}=|A|\delta(\nu/|A|).
\]
The estimate also covers the streamwise mean.  Write \(\overline{\bm v}:=\mathscr P_0\bm v\).  Averaging the divergence constraint gives \(\partial_y\overline v_2=0\), and the wall condition implies \(\overline v_2=0\).  The mean horizontal component solves the one-dimensional Dirichlet heat equation with viscosity \(\nu\) and is therefore exponentially damped. 
\end{proof}

Define the bounded charge-to-velocity map
\[
 \mathcal B_{\mathrm{ci}}(a_+,a_-)
 :=\gamma E_0\Pdiv\bigl((a_+-a_-)\bm e_1\bigr).
\]
With respect to
\(
 \Hphase=L^2_\sigma(\Omega)\times L^2(\Omega)^2,
\)
the linearized generator has the triangular form
\[
 \Lzero=
 \begin{pmatrix}
  \LC & \mathcal B_{\mathrm{ci}}\\
  0 & \Li
 \end{pmatrix},
 \qquad
 \mathcal D(\Lzero)=\Dphase.
\]

The diagonal decay estimates and the bounded charge-to-velocity coupling now yield Theorem~\ref{thm:linear-main}.

\begin{proof}[Proof of Theorem~\ref{thm:linear-main}]
The block-diagonal operator \(\operatorname{diag}(\LC,\Li)\) is sectorial on \(\Hphase\) with domain \(\Dphase\), and the off-diagonal map \(\mathcal B_{\mathrm{ci}}\) is bounded on the base space.  The bounded perturbation theorem therefore yields sectoriality of \(\Lzero\) on the same domain.  Put
\[
 \omega_{\mathrm{ion}}:=D_{\min}\lambda_1.
\]
Lemma~\ref{lem:ionic} gives
\[
 \norm{(a_+,a_-)(t)}_{L^2\times L^2}
 \le \left(\frac{D_{\max}}{D_{\min}}\right)^{1/2}
 \e^{-\omega_{\mathrm{ion}}t}
 \norm{(a_+,a_-)(0)}_{L^2\times L^2}.
\]
By Lemma~\ref{lem:couette-semigroup}, the projected velocity satisfies
\[
 \bm v(t)=\e^{t\LC}\bm v_{\mathrm{in}}
 +\gamma E_0\int_0^t
 \e^{(t-s)\LC}\Pdiv\bigl((a_+(s)-a_-(s))\bm e_1\bigr)\dd s.
\]
Fix \(0<\omega_0<\min\{\omega_{\mathrm C},\omega_{\mathrm{ion}}\}\).  Then
\[
 \int_0^t \e^{-\omega_{\mathrm C}(t-s)}\e^{-\omega_{\mathrm{ion}}s}\dd s
 \le C_{\omega_0}\e^{-\omega_0t},
 \qquad t\ge0,
\]
with \(C_{\omega_0}\) independent of \(t\).  The boundedness of \(\Pdiv\) on \(L^2\), together with the velocity and ionic estimates, proves \eqref{eq:main-linear-decay}.
\end{proof}

A fixed graph norm is needed for the interpolation argument.  The next estimate identifies the graph domain of the full linearized generator with the natural \(H^2\)-product space.

\begin{lemma}[Equivalence of graph norms]\label{lem:graph-norm}
There is a constant \(C_{\mathrm g}\ge1\), depending on the fixed physical parameters, such that
\begin{equation}\label{eq:graph-norm-equivalence}
 C_{\mathrm g}^{-1}\|\bm Z\|_{\Dphase}
 \le \|\bm Z\|_{\Hphase}+\|\Lzero\bm Z\|_{\Hphase}
 \le C_{\mathrm g}\|\bm Z\|_{\Dphase},
 \qquad \bm Z\in\Dphase.
\end{equation}
\end{lemma}

\begin{proof}
The definitions of \(\LC\), \(\Li\) and the bounded coupling \(\mathcal B_{\mathrm{ci}}\) give the upper bound.  For the reverse bound, fix once and for all a number
\[
 0<\varepsilon_{\mathrm g}<\frac12.
\]
Write
\[
 \Li=-\operatorname{diag}(D_+\AD,D_-\AD)+\mathcal K_{\mathrm{ion}},
\]
where \(\mathcal K_{\mathrm{ion}}\) contains the Couette transports, the tangential electric drifts, and the bounded Poisson reaction matrix.  Elliptic interpolation gives
\[
 \|\mathcal K_{\mathrm{ion}}\bm a\|_{L^2\times L^2}
 \le \varepsilon_{\mathrm g}
 \|\operatorname{diag}(D_+\AD,D_-\AD)\bm a\|_{L^2\times L^2}
 +C_{\varepsilon_{\mathrm g}}\|\bm a\|_{L^2\times L^2},
 \qquad \bm a=(a_+,a_-).
\]
Since
\[
 \operatorname{diag}(D_+\AD,D_-\AD)\bm a
 =-\Li\bm a+\mathcal K_{\mathrm{ion}}\bm a,
\]
we obtain
\[
 (1-\varepsilon_{\mathrm g})
 \|\operatorname{diag}(D_+\AD,D_-\AD)\bm a\|_{L^2\times L^2}
 \le \|\Li\bm a\|_{L^2\times L^2}
 +C_{\varepsilon_{\mathrm g}}\|\bm a\|_{L^2\times L^2}.
\]
Because \(1-\varepsilon_{\mathrm g}>1/2\), the Dirichlet elliptic estimate yields
\[
 \|a_+\|_{H^2}+\|a_-\|_{H^2}
 \le C\bigl(
 \|\Li(a_+,a_-)\|_{L^2\times L^2}
 +\|a_+\|_{L^2}+\|a_-\|_{L^2}\bigr).
\]
For the velocity block, write
\[
 \LC=-\nu\AS+\mathcal K_{\mathrm C},
 \qquad
 \mathcal K_{\mathrm C}\bm v
 =-\Pdiv(Ay\partial_x\bm v+Av_2\bm e_1).
\]
The Stokes elliptic estimate and first-order interpolation, with the same
\(\varepsilon_{\mathrm g}\), give
\[
 \|\mathcal K_{\mathrm C}\bm v\|_{L^2}
 \le \varepsilon_{\mathrm g}\nu\|\AS\bm v\|_{L^2}
 +C_{\varepsilon_{\mathrm g}}\|\bm v\|_{L^2}.
\]
Using \(\nu\AS\bm v=-\LC\bm v+\mathcal K_{\mathrm C}\bm v\), we obtain
\[
 (1-\varepsilon_{\mathrm g})\nu\|\AS\bm v\|_{L^2}
 \le \|\LC\bm v\|_{L^2}
 +C_{\varepsilon_{\mathrm g}}\|\bm v\|_{L^2}.
\]
The Stokes elliptic estimate now gives
\[
 \|\bm v\|_{H^2}
 \le C\bigl(\|\LC\bm v\|_{L^2}+\|\bm v\|_{L^2}\bigr).
\]
For \(\bm Z=(\bm v,a_+,a_-)\),
\[
 \Li(a_+,a_-)
\]
is the ionic component of \(\Lzero\bm Z\), while
\[
 \LC\bm v
 =(\Lzero\bm Z)_{\mathrm{fluid}}
 -\mathcal B_{\mathrm{ci}}(a_+,a_-).
\]
Since \(\mathcal B_{\mathrm{ci}}\) is bounded on \(L^2(\Omega)^2\), the last three estimates imply the lower bound in \eqref{eq:graph-norm-equivalence}.
\end{proof}

\section{Nonlinear stability}\label{sec:nonlinear-proof}

We write the nonlinear perturbation equations in the variables
\[
 \bm Z=(\bm v,a_+,a_-),
 \qquad
 q=a_+-a_-,
 \qquad
 \varphi=\Kop q.
\]
The equations read
\begin{equation}\label{eq:nonlinear-system}
\left\{
\begin{aligned}
 \partial_t\bm v+Ay\partial_x\bm v+Av_2\bm e_1+\nabla\pi
 &=\nu\Delta\bm v+\gamma E_0q\bm e_1
 -\bm v\cdot\nabla\bm v-\gamma q\nabla\varphi,\\
 \Div\bm v&=0,\\
 \partial_t a_+ +Ay\partial_xa_+
 &=D_+\Delta a_+ +D_+c_0\Delta\varphi-D_+E_0\partial_xa_+\\
 &\quad-\bm v\cdot\nabla a_+ +D_+\Div(a_+\nabla\varphi),\\
 \partial_t a_- +Ay\partial_xa_-
 &=D_-\Delta a_- -D_-c_0\Delta\varphi+D_-E_0\partial_xa_-\\
 &\quad-\bm v\cdot\nabla a_- -D_-\Div(a_-\nabla\varphi),\\
 -\lambda_{\mathrm D}^2\Delta\varphi&=a_+-a_-,
 \qquad \varphi|_{y=\pm1}=0.
\end{aligned}
\right.
\end{equation}
Applying \(\Pdiv\) to the momentum equation, we obtain
\[
 \partial_t\bm Z=\Lzero\bm Z+\Nzero(\bm Z),
\]
where
\[
 \Nzero(\bm Z)=
 \begin{pmatrix}
 -\Pdiv(\bm v\cdot\nabla\bm v)-\gamma\Pdiv(q\nabla\Kop q)\\[1mm]
 -\bm v\cdot\nabla a_+ +D_+\Div(a_+\nabla\Kop q)\\[1mm]
 -\bm v\cdot\nabla a_- -D_-\Div(a_-\nabla\Kop q)
 \end{pmatrix}.
\]

\begin{lemma}[Quadratic mapping property]\label{lem:quadratic}
Let \(\theta\in(3/4,1)\).  The map \(\Nzero\) belongs to \(C^1(\Xtheta,\Hphase)\) and satisfies
\begin{equation}\label{eq:quadratic}
 \Nzero(0)=0,
 \qquad
 D\Nzero(0)=0,
 \qquad
 \norm{\Nzero(\bm Z)}_{\Hphase}
 \le C\norm{\bm Z}_{\Xtheta}^2.
\end{equation}
For all \(\bm Z,\widetilde{\bm Z}\in\Xtheta\),
\begin{equation}\label{eq:quadratic-difference}
 \norm{\Nzero(\bm Z)-\Nzero(\widetilde{\bm Z})}_{\Hphase}
 \le C\bigl(\|\bm Z\|_{\Xtheta}+\|\widetilde{\bm Z}\|_{\Xtheta}\bigr)
 \norm{\bm Z-\widetilde{\bm Z}}_{\Xtheta}.
\end{equation}
In particular, its Lipschitz constant on a ball of radius \(R\) is at most \(2CR\), with \(C\) independent of \(R\).
\end{lemma}

\begin{proof}
Put \(r_\theta:=2\theta>3/2\).  The embedding \eqref{eq:Xembed} gives
\[
 \|\bm v\|_{H^{r_\theta}}
 +\|a_+\|_{H^{r_\theta}}
 +\|a_-\|_{H^{r_\theta}}
 \le C\|\bm Z\|_{\Xtheta}.
\]
Since \(q=a_+-a_-\), Dirichlet elliptic regularity yields
\[
 \|\Kop q\|_{H^{\sigma+2}}
 \le C\lambda_{\mathrm D}^{-2}\|q\|_{H^\sigma},
 \qquad 0\le\sigma\le r_\theta.
\]
In two dimensions,
\(
 H^{r_\theta}\hookrightarrow L^\infty
\)
and
\(
 H^{r_\theta+2}\hookrightarrow W^{1,\infty}
\).
Consequently,
\[
 \|\bm v\cdot\nabla\bm v\|_{L^2}
 \le C\|\bm v\|_{H^{r_\theta}}^2,
 \qquad
 \|q\nabla\Kop q\|_{L^2}
 \le C\lambda_{\mathrm D}^{-2}\|q\|_{H^{r_\theta}}^2.
\]
For each sign,
\[
 \|\bm v\cdot\nabla a_\pm\|_{L^2}
 \le C\|\bm v\|_{H^{r_\theta}}\|a_\pm\|_{H^{r_\theta}},
\]
and
\begin{align*}
 \|\Div(a_\pm\nabla\Kop q)\|_{L^2}
 &\le \|\nabla a_\pm\|_{L^2}\|\nabla\Kop q\|_{L^\infty}
 +\|a_\pm\|_{L^\infty}\|\Delta\Kop q\|_{L^2}\\
 &\le C\lambda_{\mathrm D}^{-2}
 \|a_\pm\|_{H^{r_\theta}}\|q\|_{H^{r_\theta}}.
\end{align*}
These estimates prove \eqref{eq:quadratic}.  More precisely, there is a continuous bilinear map \(\mathcal Q:\Xtheta\times\Xtheta\longrightarrow\Hphase\) with \(\Nzero(\bm Z)=\mathcal Q(\bm Z,\bm Z)\).  Expanding
\[
 \mathcal Q(\bm Z,\bm Z)-\mathcal Q(\widetilde{\bm Z},\widetilde{\bm Z})
 =\mathcal Q(\bm Z-\widetilde{\bm Z},\bm Z)
 +\mathcal Q(\widetilde{\bm Z},\bm Z-\widetilde{\bm Z})
\]
gives \eqref{eq:quadratic-difference}.  Its derivative is
\(D\Nzero(\bm Z)\bm W=\mathcal Q(\bm W,\bm Z)+\mathcal Q(\bm Z,\bm W)\),
which is continuous in \(\bm Z\) and vanishes at \(\bm Z=0\).
\end{proof}

Lemma~\ref{lem:graph-norm} identifies \(\Xtheta\), up to equivalent norms, with
\(
[\Hphase,\mathcal D(\Lzero)]_\theta
\).
Analyticity gives
\[
 \|\e^{t\Lzero}\|_{\mathcal L(\Hphase,\mathcal D(\Lzero))}
 \le Ct^{-1},
 \qquad 0<t\le1,
\]
while exponential stability and the commutation
\(
\Lzero\e^{t\Lzero}x=\e^{t\Lzero}\Lzero x
\)
for \(x\in\mathcal D(\Lzero)\) give exponential decay on both \(\Hphase\) and the graph domain for \(t\ge1\).  Interpolating the domain and range spaces and decreasing the decay rate if necessary, yields a number \(\omega_1>0\) for which
\begin{equation}\label{eq:semigroup-estimates}
\left\{
\begin{aligned}
 \norm{\e^{t\Lzero}}_{\mathcal L(\Hphase,\Xtheta)}
 &\le Ct^{-\theta}\e^{-\omega_1t},
 &&t>0,\\
 \norm{\e^{t\Lzero}}_{\mathcal L(\Xtheta)}
 &\le C\e^{-\omega_1t},
 &&t\ge0.
\end{aligned}
\right.
\end{equation}
Let \(\theta<\beta<1\) and set
\(
 \mathscr X_\beta:=[\Hphase,\Dphase]_\beta
\).
Interpolation also gives the difference estimate used in the positive-time regularity argument.  Indeed,
\[
 \|\e^{h\Lzero}-\Id\|_{\mathcal L(\Hphase)}
 \le C,
 \qquad
 \|(\e^{h\Lzero}-\Id)x\|_{\Hphase}
 \le Ch\|x\|_{\Dphase},
 \qquad 0<h\le1.
\]
Interpolating these two bounds gives
\[
 \|(\e^{h\Lzero}-\Id)x\|_{\Hphase}
 \le Ch^\beta\|x\|_{\mathscr X_\beta}.
\]
On the other hand, \(\e^{h\Lzero}-\Id\) is uniformly bounded on \(\mathscr X_\beta\).  Interpolation of the range spaces, together with the complex-interpolation reiteration theorem \cite[Section~4.6]{BerghLofstrom1976},
\(
[\Hphase,\mathscr X_\beta]_{\theta/\beta}=\Xtheta
\)
(which applies because \(0<\theta/\beta<1\)),
then yields
\begin{equation}\label{eq:interpolation-difference}
 \|(\e^{h\Lzero}-\Id)x\|_{\Xtheta}
 \le C h^{\beta-\theta}\|x\|_{\mathscr X_\beta},
 \qquad 0<h\le1.
\end{equation}

The following analytic-semigroup regularity lemma will be used in both nonlinear stability proofs.

\begin{lemma}[H\"older forcing and the graph domain]\label{lem:holder-forcing}
Let \(\mathcal G\) generate an analytic \(C_0\)-semigroup \(\mathcal S(t)\) on a Banach space \(\mathscr V\).  Suppose \(\bm F\in C^\alpha([a,b];\mathscr V)\) for some \(0<\alpha<1\), and let \(\bm w_a\in\mathscr V\).  Then
\[
 \bm w(t)=\mathcal S(t-a)\bm w_a
 +\int_a^t\mathcal S(t-s)\bm F(s)\dd s
\]
belongs to \(C((a,b];\mathcal D(\mathcal G))\cap C^1((a,b);\mathscr V)\), where the domain carries its graph norm, and satisfies \(\bm w'=\mathcal G\bm w+\bm F\) on \((a,b)\).
\end{lemma}

\begin{proof}
Analyticity gives \(\|\mathcal G\mathcal S(r)\|_{\mathcal L(\mathscr V)}\le C/r\) for \(r>0\).  Fix \(t\in(a,b]\) and write \(\bm F(s)=\bm F(t)+[\bm F(s)-\bm F(t)]\).  For the constant part,
\[
 \mathcal G\int_a^t\mathcal S(t-s)\bm F(t)\dd s
 =[\mathcal S(t-a)-\Id]\bm F(t).
\]
For the H\"older remainder define
\[
 I_\varepsilon(t)=\int_a^{t-\varepsilon}\mathcal S(t-s)[\bm F(s)-\bm F(t)]\dd s,
 \qquad 0<\varepsilon<t-a.
\]
Then \(I_\varepsilon(t)\in\mathcal D(\mathcal G)\) and
\[
 \mathcal G I_\varepsilon(t)
 =\int_a^{t-\varepsilon}\mathcal G\mathcal S(t-s)[\bm F(s)-\bm F(t)]\dd s,
\]
while
\[
 \|\mathcal G\mathcal S(t-s)[\bm F(s)-\bm F(t)]\|_{\mathscr V}
 \le C[\bm F]_{C^\alpha}(t-s)^{\alpha-1}.
\]
Since \(\alpha>0\), both \(I_\varepsilon(t)\) and \(\mathcal G I_\varepsilon(t)\) converge in \(\mathscr V\) as \(\varepsilon\to0\).  Closedness of \(\mathcal G\) therefore gives
\begin{equation}\label{eq:holder-generator-identity}
 \mathcal G\bm w(t)
 =\mathcal G\mathcal S(t-a)\bm w_a
 +[\mathcal S(t-a)-\Id]\bm F(t)
 +\int_a^t\mathcal G\mathcal S(t-s)[\bm F(s)-\bm F(t)]\dd s.
\end{equation}
The integral in \eqref{eq:holder-generator-identity} is continuous in \(t\): away from \(s=t\) this follows from strong continuity, while the contribution from \(0<t-s<\delta\) is bounded by \(C\delta^\alpha\).  Hence \(\bm w\in C((a,b];\mathcal D(\mathcal G))\).  Finally,
\[
 \frac{\mathcal S(h)-\Id}{h}\bm z
 =\frac1h\int_0^h\mathcal S(r)\mathcal G\bm z\dd r,
 \qquad \bm z\in\mathcal D(\mathcal G),
\]
together with the mild increment formula gives the right and left derivatives
\(\bm w'=\mathcal G\bm w+\bm F\).  Their continuity follows from the graph-norm continuity just proved.
\end{proof}

We next apply the preceding semigroup estimates and Lemma~\ref{lem:holder-forcing} to prove Theorem~\ref{thm:nonlinear-main}.

\begin{proof}[Proof of Theorem~\ref{thm:nonlinear-main}]
Fix \(0<\omega<\omega_1\) and define
\[
 \mathscr Y_\omega
 :=\left\{\bm Z\in C([0,\infty);\Xtheta):
 \sup_{t\ge0}\e^{\omega t}\|\bm Z(t)\|_{\Xtheta}<\infty\right\}.
\]
Equip this space with the norm
\[
 \|\bm Z\|_{\mathscr Y_\omega}
 :=\sup_{t\ge0}\e^{\omega t}\|\bm Z(t)\|_{\Xtheta}.
\]
With this weighted supremum norm, \(\mathscr Y_\omega\) is a Banach space.  For \(\bm Z\in\mathscr Y_\omega\), let
\[
 (\Phi\bm Z)(t)
 :=\e^{t\Lzero}\bm Z_{\mathrm{in}}
 +\int_0^t\e^{(t-s)\Lzero}\Nzero(\bm Z(s))\dd s.
\]
The map \(\Phi\) is well defined on \(\mathscr Y_\omega\).  Indeed, \(\Nzero(\bm Z(\cdot))\) is continuous as an \(\Hphase\)-valued map, and the singularity \((t-s)^{-\theta}\) in \eqref{eq:semigroup-estimates} is integrable because \(\theta<1\).  Together, strong continuity of the semigroup on \(\Xtheta\) and the estimate
\[
 \left\|
 \int_0^t \e^{(t-s)\Lzero}\Nzero(\bm Z(s))\dd s
 \right\|_{\Xtheta}
 \le C t^{1-\theta}
 \sup_{0\le s\le t}\|\Nzero(\bm Z(s))\|_{\Hphase},
\]
show that \(\Phi\bm Z\in C([0,\infty);\Xtheta)\).  Here strong continuity on \(\Xtheta\) follows first on the dense subspace \(\Dphase\) and then on all of \(\Xtheta\) from the uniform short-time bound in \eqref{eq:semigroup-estimates}.

By Lemma~\ref{lem:quadratic} and \eqref{eq:semigroup-estimates},
\[
 \|\Phi\bm Z\|_{\mathscr Y_\omega}
 \le C_0\|\bm Z_{\mathrm{in}}\|_{\Xtheta}
 +C_1\|\bm Z\|_{\mathscr Y_\omega}^2,
\]
because
\[
 \int_0^\infty r^{-\theta}
 \e^{-(\omega_1-\omega)r}\dd r<\infty.
\]
Likewise, on the closed ball of radius \(R\),
\[
 \|\Phi\bm Z-\Phi\widetilde{\bm Z}\|_{\mathscr Y_\omega}
 \le C_2R\|\bm Z-\widetilde{\bm Z}\|_{\mathscr Y_\omega}.
\]
Choose \(R=2C_0\|\bm Z_{\mathrm{in}}\|_{\Xtheta}\).  After decreasing \(\varepsilon_*\) so that \(C_1R\le1/2\) and \(C_2R<1\), the map \(\Phi\) sends the closed ball of radius \(R\) into itself and is a strict contraction there.  Both \(\e^{t\Lzero}\) and \(\Nzero\) preserve real-valued states, so \(\Phi\) preserves the real subspace of \(\mathscr Y_\omega\).  Hence real-valued initial data produce a real-valued fixed point.  Banach's fixed-point theorem therefore gives a global mild solution in \(\mathscr Y_\omega\), and the definition of the weighted norm yields \eqref{eq:main-nonlinear-decay}.

We next verify uniqueness in the full class stated in Theorem~\ref{thm:nonlinear-main}.  Let \(\bm Z_1,\bm Z_2\in C([0,T];\Xtheta)\) be two mild solutions with the same initial data, and define
\[
 M_T:=\max_{j=1,2}\sup_{0\le t\le T}\|\bm Z_j(t)\|_{\Xtheta}.
\]
Lemma~\ref{lem:quadratic} and \eqref{eq:semigroup-estimates} then give
\[
 \|\bm Z_1(t)-\bm Z_2(t)\|_{\Xtheta}
 \le C M_T\int_0^t (t-s)^{-\theta}
 \|\bm Z_1(s)-\bm Z_2(s)\|_{\Xtheta}\dd s.
\]
Taking the supremum on \([0,t]\) gives
\[
 \sup_{0\le s\le t}\|\bm Z_1(s)-\bm Z_2(s)\|_{\Xtheta}
 \le \frac{CM_T}{1-\theta}\,t^{1-\theta}
 \sup_{0\le s\le t}\|\bm Z_1(s)-\bm Z_2(s)\|_{\Xtheta}.
\]
Choose \(\delta_T>0\) so that \(CM_T\delta_T^{1-\theta}/(1-\theta)\le1/2\).  The two solutions coincide on \([0,\delta_T]\).  Because both solutions remain in the same \(\Xtheta\)-ball of radius \(M_T\), the local Lipschitz constant is unchanged when the mild equation is restarted.  Finitely many intervals of length \(\delta_T\) therefore cover \([0,T]\).  This proves uniqueness on every bounded time interval.

To obtain positive-time regularity, fix \(\beta\in(\theta,1)\).  The estimate
\[
 \|\e^{t\Lzero}\|_{\mathcal L(\Hphase,\mathscr X_\beta)}
 \le C_\beta t^{-\beta}\e^{-\omega_1t}
\]
and the variation-of-constants formula restarted at \(t_0/2>0\) show that
\[
 \sup_{t\in[t_0,T]}\|\bm Z(t)\|_{\mathscr X_\beta}<\infty
 \qquad (0<t_0<T<\infty).
\]
For \(t,t+h\in[t_0,T]\), write
\[
 \bm Z(t+h)-\bm Z(t)
 =(\e^{h\Lzero}-\Id)\bm Z(t)
 +\int_t^{t+h}\e^{(t+h-s)\Lzero}
 \Nzero(\bm Z(s))\dd s.
\]
The first term on the right-hand side is bounded by \(C_{t_0,T}h^{\beta-\theta}\) through
\eqref{eq:interpolation-difference}.  For the integral term, the first estimate in
\eqref{eq:semigroup-estimates} and boundedness of
\(\Nzero(\bm Z(\cdot))\) on \([t_0,T]\) give
\[
 \left\|
 \int_t^{t+h}\e^{(t+h-s)\Lzero}
 \Nzero(\bm Z(s))\dd s
 \right\|_{\Xtheta}
 \le C_{t_0,T}h^{1-\theta}.
\]
Because \(\beta<1\), one has \(\beta-\theta<1-\theta\), and therefore
\[
 \|\bm Z(t+h)-\bm Z(t)\|_{\Xtheta}
 \le C_{t_0,T}h^{\beta-\theta}.
\]
It follows that \(t\mapsto\Nzero(\bm Z(t))\) is locally H\"older continuous as an
\(\Hphase\)-valued function.  Apply Lemma~\ref{lem:holder-forcing} on a compact interval starting at a positive time, with \(\mathcal G=\Lzero\) and \(\bm F=\Nzero(\bm Z)\).  Since the starting time can be chosen below any given \(t>0\), it follows that \(\bm Z(t)\in\mathcal D(\Lzero)\) for all \(t>0\) and \(\bm Z'\in C((0,\infty);\Hphase)\).  Moreover, \(\Nzero(\bm Z(\cdot))\in C((0,\infty);\Hphase)\), so
\[
 \Lzero\bm Z=\bm Z'-\Nzero(\bm Z)
\]
is continuous as an \(\Hphase\)-valued function.  Therefore \(\bm Z\) is continuous in the graph norm of \(\Lzero\), and Lemma~\ref{lem:graph-norm} yields
\[
 \bm Z\in C((0,\infty);\Dphase)
 \cap C^1((0,\infty);\Hphase).
\]

Finally, recall
\[
 c_\pm=c_0+a_\pm.
\]
Combining the embedding \eqref{eq:Xembed} with \eqref{eq:main-nonlinear-decay} yields
\[
 \|a_\pm(t)\|_{L^\infty}
 \le C_{\mathrm{emb}}C\e^{-\omega t}
 \|\bm Z_{\mathrm{in}}\|_{\Xtheta}.
\]
Choosing the initial-data threshold additionally so that
\(C_{\mathrm{emb}}C\varepsilon_*\le c_0/2\) ensures
\(c_\pm\ge c_0/2\) for \(t\ge0\).  Since
\[
 \bm Z(t)\in\Dphase,\qquad \Nzero(\bm Z(t))\in\Hphase,\qquad t>0,
\]
we apply \(\Id-\Pdiv\) to the unprojected momentum equation and obtain
\[
 \nabla\pi
 =(\Id-\Pdiv)\Bigl[
 \nu\Delta\bm v-Ay\partial_x\bm v-Av_2\bm e_1
 +\gamma E_0q\bm e_1
 -\bm v\cdot\nabla\bm v-\gamma q\nabla\varphi
 \Bigr]\in L^2(\Omega;\R^2).
\]
For \(t>0\), the Helmholtz decomposition determines a periodic
\(\pi(t)\in H^1(\Omega)\) up to a spatial constant.  We select the unique representative satisfying
\[
 \int_\Omega\pi(t)\,\dd x\dd y=0,\qquad t>0.
\]
The divergence constraint is built into \(\Xtheta\), while \(-\lambda_{\mathrm D}^2\Delta\varphi=q\) holds by the definition \(\varphi=\Kop q\).  Accordingly, the projected mild solution satisfies the full perturbation system \eqref{eq:nonlinear-system} in \(L^2\) for \(t>0\), with the boundary conditions understood in the trace sense.
\end{proof}

\section{Ionic enhanced dissipation and equal-diffusivity factorization}\label{sec:ED-proof}

This section quantifies the decay of the nonzero streamwise ionic modes in the strong-shear regime and explains the sharper structure that appears when the two diffusivities coincide.  Enhanced dissipation near Couette flow is well established for hydrodynamic perturbations.  Wei and Zhang proved nonlinear enhanced dissipation for the two-dimensional Couette problem in \cite{WeiZhang2023}, and later obtained a finite-channel stability threshold with enhanced-dissipation estimates under Navier-slip boundary conditions in \cite{WeiZhang2026}.  In the present ionic problem, the scalar Dirichlet Couette estimate supplies the mixing input.  For unequal diffusivities, the two species have different scalar generators and the Poisson reaction is controlled by Duhamel's formula.  When \(D_+=D_-\), the common scalar generator can instead be separated from a contractive drift--reaction factor.  We first derive the scalar nonzero-mode estimate used in both cases and then prove the two conclusions of Theorem~\ref{thm:ED-main}.

\begin{proof}[Proof of Theorem~\ref{thm:ED-main}]
Recall
\[
 \kappa_{\mathrm D}=\frac{c_0}{\lambda_{\mathrm D}^2},
 \qquad
 \bm a=\begin{pmatrix}a_+\\ a_-\end{pmatrix}.
\]
Define the diagonal advection--diffusion operator
\[
 \mathscr S_0
 :=\begin{pmatrix}
 D_+\Delta-Ay\partial_x-D_+E_0\partial_x&0\\
 0&D_-\Delta-Ay\partial_x+D_-E_0\partial_x
 \end{pmatrix},
\]
with homogeneous Dirichlet conditions at \(y=\pm1\), and define the reaction matrix
\[
 \mathcal R_{\mathrm D}
 :=\kappa_{\mathrm D}
 \begin{pmatrix}
 -D_+&D_+\\
 D_-&-D_-
 \end{pmatrix}.
\]
Since
\(
\Delta\varphi=-(a_+-a_-)/\lambda_{\mathrm D}^2
\),
the ionic linearization is
\begin{equation}\label{eq:ED-unequal-system}
 \partial_t\bm a=\mathscr S_0\bm a+\mathcal R_{\mathrm D}\bm a.
\end{equation}

Tangential drifts preserve the scalar decay rate.  Indeed, if
\(
T_c(t)f(x,y):=f(x-ct,y)
\),
then \(T_c(t)\) is unitary on \(L^2(\Omega)\), preserves the Dirichlet condition in \(y\), and commutes with \(\Delta\) and \(y\partial_x\).  Hence
\[
 \e^{t(D_\pm\Delta-Ay\partial_x\mp D_\pm E_0\partial_x)}
 =
 T_{\pm D_\pm E_0}(t)\,
 \e^{t(D_\pm\Delta-Ay\partial_x)}.
\]
Consequently, the scalar \(L^2\)-decay constants are independent of \(E_0\).  These translations are used only to estimate the diagonal semigroups.  The coupled Duhamel equation below remains in the original common coordinates, so the reaction matrix stays constant.

Because all coefficients are independent of \(x\), the projection \(\mathscr P_{\neq 0}\) from Section~\ref{sec:main} commutes with both \(\mathscr S_0\) and \(\mathcal R_{\mathrm D}\).  The shear leaves the streamwise mean unchanged, while Lemma~\ref{lem:ionic} controls that mode through diffusion and the ionic reaction.

After removing the constant scalar drift, define
\[
 \tau:=|A|t,\qquad \varepsilon_\pm:=\frac{D_\pm}{|A|}.
\]
A diagonal equation then becomes
\[
 \partial_\tau f+\operatorname{sgn}(A)y\,\partial_xf
 -\varepsilon_\pm\Delta f=0.
\]
For \(A<0\), reflect \(x\mapsto-x\).  To match the unit torus and unit-height channel used in \cite{AlbrittonBeekieNovack2022}, define
\[
 X:=\frac{x}{2\pi},\qquad Y:=\frac{y+1}{2},\qquad
 g(X,Y,\tau):=f(2\pi X,2Y-1,\tau),
\]
and set
\[
 b(Y):=\frac{2Y-1}{2\pi},\qquad
 \eta_\pm:=\frac{\varepsilon_\pm}{4},\qquad
 \xi_\pm:=\frac{\varepsilon_\pm}{4\pi^2}.
\]
In these variables, the equation on \(\T_1\times(0,1)\) is
\[
 \partial_\tau g+b(Y)\partial_Xg
 -\eta_\pm\partial_Y^2g-\xi_\pm\partial_X^2g=0,
 \qquad g|_{Y=0,1}=0.
\]
We separate the horizontal heat flow from the vertically diffusive Couette evolution.  For a vertical diffusivity \(\eta>0\), let \(\mathcal V_\eta(\tau)\) denote the Dirichlet semigroup generated by
\[
 \partial_\tau h+b(Y)\partial_Xh-\eta\partial_Y^2h=0.
\]
Since \(b'(Y)=1/\pi\), the monotone case \(N=0\) of Theorem~1.1 in \cite{AlbrittonBeekieNovack2022} provides constants \(C_{\mathrm A}\ge1\), \(d_{\mathrm A}>0\), and \(\eta_0>0\) such that
\[
 \|\mathcal V_\eta(\tau)h_{\mathrm{in}}\|_{L^2}
 \le C_{\mathrm A}\e^{-d_{\mathrm A}\eta^{1/3}\tau}
       \|h_{\mathrm{in}}\|_{L^2},
 \qquad 0<\eta<\eta_0,
\]
for data with zero \(X\)-average.  The horizontal heat semigroup commutes with \(\mathcal V_\eta\) and is an \(L^2\)-contraction.  Fourier expansion in \(X\) therefore gives the exact factorization
\[
 g(\tau)=\e^{\xi_\pm\tau\partial_X^2}
          \mathcal V_{\eta_\pm}(\tau)g(0).
\]
Consequently, the preceding scalar estimate applies to \(g\) with the same constants.  The factorization first holds for smooth data and extends to \(L^2\) data by density.  Since
\[
 \|g(\tau)\|_{L^2(\T_1\times(0,1))}^2
 =\frac{1}{4\pi}\,
 \|f(t)\|_{L^2(\T_{2\pi}\times(-1,1))}^2,
 \qquad \tau=|A|t,
\]
and
\[
 \|g(0)\|_{L^2(\T_1\times(0,1))}^2
 =\frac{1}{4\pi}\,
 \|f(0)\|_{L^2(\T_{2\pi}\times(-1,1))}^2,
\]
the Jacobian factor cancels in the decay estimate.

We take the constants from Theorem~\ref{thm:ED-main} to be
\[
 \varepsilon_{\mathrm{ED}}=2\eta_0,\qquad
 C_{\mathrm{ED}}=C_{\mathrm A},\qquad
 d_{\mathrm{ED}}=4^{-1/3}d_{\mathrm A}.
\]
Then \(D_{\max}/|A|\le\varepsilon_{\mathrm{ED}}\) implies \(0<\eta_\pm<\eta_0\).  Returning to \(t\) gives
\[
 d_{\mathrm A}\eta_\pm^{1/3}\tau
 =d_{\mathrm{ED}}D_\pm^{1/3}|A|^{2/3}t.
\]
Define
\[
 \mathscr U_0(t):=\e^{t\mathscr S_0},
 \qquad
 \Lambda_{\mathrm{ED}}:=d_{\mathrm{ED}}D_{\min}^{1/3}|A|^{2/3}.
\]
The slower of the two diagonal rates then gives
\begin{equation}\label{eq:diagonal-ED}
 \norm{\mathscr U_0(t)\mathscr P_{\neq 0}}_{\mathcal L(L^2(\Omega)^2)}
 \le C_{\mathrm{ED}}\e^{-\Lambda_{\mathrm{ED}}t}.
\end{equation}

On \(L^2(\Omega)^2\), the reaction matrix satisfies
\begin{equation}\label{eq:reaction-bound}
 \norm{\mathcal R_{\mathrm D}}_{\mathcal L(L^2(\Omega)^2)}
 \le
 \kappa_{\mathrm D}
 \sqrt{2(D_+^2+D_-^2)}
 \le 2\kappa_{\mathrm D}D_{\max}.
\end{equation}
Applying \(\mathscr P_{\neq 0}\) to \eqref{eq:ED-unequal-system} and using variation of constants yields
\[
 \bm a_{\neq 0}(t)
 =\mathscr U_0(t)\bm a_{\neq 0}(0)
 +\int_0^t
 \mathscr U_0(t-s)\mathcal R_{\mathrm D}
 \bm a_{\neq 0}(s)\dd s.
\]
By \eqref{eq:diagonal-ED}--\eqref{eq:reaction-bound}, we have
\[
 \e^{\Lambda_{\mathrm{ED}}t}
 \norm{\bm a_{\neq 0}(t)}_{L^2(\Omega)^2}
 \le
 C_{\mathrm{ED}}\norm{\bm a_{\neq 0}(0)}_{L^2(\Omega)^2}
 +2C_{\mathrm{ED}}\kappa_{\mathrm D}D_{\max}
 \int_0^t
 \e^{\Lambda_{\mathrm{ED}}s}
 \norm{\bm a_{\neq 0}(s)}_{L^2(\Omega)^2}\dd s.
\]
Gronwall's inequality gives
\begin{equation}\label{eq:ED-before-absorption}
 \norm{\bm a_{\neq 0}(t)}_{L^2(\Omega)^2}
 \le C_{\mathrm{ED}}
 \exp\!\left[
 -\bigl(
 \Lambda_{\mathrm{ED}}
 -2C_{\mathrm{ED}}\kappa_{\mathrm D}D_{\max}
 \bigr)t
 \right]
 \norm{\bm a_{\neq 0}(0)}_{L^2(\Omega)^2}.
\end{equation}
The second assumption in \eqref{eq:ED-strong-shear} is exactly
\[
 2C_{\mathrm{ED}}\kappa_{\mathrm D}D_{\max}
 \le \frac12\Lambda_{\mathrm{ED}},
\]
so \eqref{eq:ED-before-absorption} reduces to \eqref{eq:main-ED}.

It remains to prove the equal-diffusivity assertion.  Suppose \(D_+=D_-=D\).  Write
\[
 \mathscr S_D:=D\Delta-Ay\partial_x,
 \qquad
 \mathcal K_E
 :=DE_0
 \begin{pmatrix}-1&0\\0&1\end{pmatrix}\partial_x
 +\kappa_{\mathrm D}D
 \begin{pmatrix}-1&1\\1&-1\end{pmatrix}.
\]
Then the ionic generator is
\[
 \mathscr S_D\Id+\mathcal K_E.
\]
For the \(k\)-th streamwise Fourier mode, \(\mathcal K_E\) is the constant matrix
\[
 M_k
 =\ii kDE_0
 \begin{pmatrix}-1&0\\0&1\end{pmatrix}
 +\kappa_{\mathrm D}D
 \begin{pmatrix}-1&1\\1&-1\end{pmatrix}.
\]
For \(\bm z=(z_+,z_-)^{\mathsf T}\in\C^2\),
\[
 \operatorname{Re}\langle M_k\bm z,\bm z\rangle_{\C^2}
 =-\kappa_{\mathrm D}D|z_+-z_-|^2\le0.
\]
Hence
\[
 \|\e^{tM_k}\|_{\mathcal L(\C^2)}\le1,
 \qquad t\ge0,
\]
and Parseval's identity gives
\[
 \|\e^{t\mathcal K_E}\|_{\mathcal L(L^2(\Omega)^2)}\le1.
\]
The operator \(\mathcal K_E\) contains no \(y\)-derivatives, so its semigroup preserves the homogeneous Dirichlet trace on \(y=\pm1\).  Its coefficients are independent of \(x\), and hence \(\e^{t\mathcal K_E}\) also commutes with \(\mathscr P_{\neq0}\).  On each streamwise Fourier mode, \(\mathscr S_D\Id\) is scalar in the species variables and therefore commutes with \(M_k\).  Thus
\[
 \e^{t(\mathscr S_D\Id+\mathcal K_E)}\mathscr P_{\neq0}
 =\e^{t\mathscr S_D}\mathscr P_{\neq0}\,\e^{t\mathcal K_E}.
\]
The scalar Couette estimate used above, applied to \(\mathscr S_D\), gives
\[
 \|\e^{t\mathscr S_D}\mathscr P_{\neq0}\|_{\mathcal L(L^2(\Omega))}
 \le C_{\mathrm{ED}}\e^{-d_{\mathrm{ED}}D^{1/3}|A|^{2/3}t}.
\]
Combining this bound with the contraction of \(\e^{t\mathcal K_E}\) proves \eqref{eq:main-ED-equal} without any restriction involving \(\kappa_{\mathrm D}\).
\end{proof}

\begin{remark}[Unequal-diffusivity threshold and factorization obstruction]\label{rem:ED-obstruction}
For \(D_+\ne D_-\), the two conditions in \eqref{eq:ED-strong-shear} are equivalent to the explicit sufficient lower bound
\[
 |A|\ge
 \max\left\{
 \frac{D_{\max}}{\varepsilon_{\mathrm{ED}}},
 \left(
 \frac{4C_{\mathrm{ED}}\kappa_{\mathrm D}D_{\max}}
      {d_{\mathrm{ED}}D_{\min}^{1/3}}
 \right)^{3/2}
 \right\}.
\]
The second term is produced by the operator-norm estimate on the reaction matrix and is not asserted to be sharp.

The exact equal-diffusivity factorization does not extend directly to \(D_+\ne D_-\).  Indeed, let
\[
 \mathsf D=
 \begin{pmatrix}D_+&0\\0&D_-\end{pmatrix},
 \qquad
 \mathsf R_{\mathrm D}
 =\kappa_{\mathrm D}
 \begin{pmatrix}-D_+&D_+\\D_-&-D_-\end{pmatrix}.
\]
The diffusive part acts through \(\mathsf D\Delta\), whereas the Poisson reaction acts through \(\mathsf R_{\mathrm D}\), and
\[
 [\mathsf D,\mathsf R_{\mathrm D}]
 =\kappa_{\mathrm D}(D_+-D_-)
 \begin{pmatrix}0&D_+\\-D_-&0\end{pmatrix}.
\]
Thus the two operators fail to commute unless \(D_+=D_-\).  The off-diagonal reaction therefore couples components evolving under different scalar diffusion operators, so the ionic generator cannot be decomposed into a common scalar Couette generator plus a commuting finite-dimensional factor.  This obstruction explains why the unequal-diffusivity proof uses Duhamel's formula.  It does not rule out a weaker sufficient condition obtained from a genuinely vector-valued hypocoercive estimate.
\end{remark}

\section{Weak Poisson--Boltzmann layers}\label{sec:PB-proof}

This section treats weak Poisson--Boltzmann layers for arbitrary fixed \(A\) and \(E_0\), with the wall-potential amplitude \(\delta_\zeta\) as the sole small parameter.  Introduce the perturbation variables
\[
 \bm u=\bm u_*+\bm v,
 \qquad
 c_\pm=c_{\pm,*}+b_\pm,
 \qquad
 \phi=\phi_*+\psi,
\]
and define
\[
 \bm Y:=(\bm v,b_+,b_-),
 \qquad
 q_\zeta:=b_+-b_-.
\]
For the perturbation potential,
\[
 -\lambda_{\mathrm D}^2\Delta\psi=q_\zeta,
 \qquad
 \psi|_{y=\pm1}=0.
\]
Linearization about the charged steady state gives
\begin{equation}\label{eq:PB-linearized}
\left\{
\begin{aligned}
 \partial_t\bm v+U_*\partial_x\bm v+v_2U_*'\bm e_1+\nabla\pi
 &=\nu\Delta\bm v
 +\gamma q_\zeta\bigl(E_0\bm e_1-\nabla\phi_*\bigr)
 -\gamma \varrho_*\nabla\psi,\\
 \Div\bm v&=0,\\
 \partial_t b_+ +U_*\partial_xb_+ +v_2c_{+,*}'
 &=D_+\Div\!\left(
 \nabla b_+ +b_+\nabla\phi_*+c_{+,*}\nabla\psi-E_0b_+\bm e_1
 \right),\\
 \partial_t b_- +U_*\partial_xb_- +v_2c_{-,*}'
 &=D_-\Div\!\left(
 \nabla b_- -b_-\nabla\phi_*-c_{-,*}\nabla\psi+E_0b_-\bm e_1
 \right),\\
 -\lambda_{\mathrm D}^2\Delta\psi&=q_\zeta,
 \qquad \psi|_{y=\pm1}=0.
\end{aligned}
\right.
\end{equation}
All perturbations have homogeneous Dirichlet traces.  Let \(\Lzeta\) be the projected operator obtained from \eqref{eq:PB-linearized} after eliminating \(\psi\).  Since the uniform problem is already realized in the coordinates \((\bm v,a_+,a_-)\), we have
\[
 \Lzeta\big|_{\zeta_+=\zeta_-=0}=\Lzero
\]
on the common domain \(\Dphase\).

To compare the charged generator with the electroneutral one, we first quantify how the steady Poisson--Boltzmann profile depends on the wall potentials.  The following estimate will be used repeatedly in the coefficient perturbation bounds.

\begin{lemma}[Small Poisson--Boltzmann profile]\label{lem:PB-small}
Fix an integer \(k\ge2\) and \(\lambda_{\mathrm D},c_0>0\).  There exist \(\delta_{\mathrm{prof}}>0\) and \(C_k>0\), with \(\delta_{\mathrm{prof}}\) allowed to depend on \(k\), such that, whenever \(\delta_\zeta\le\delta_{\mathrm{prof}}\),
\begin{equation}\label{eq:PB-small}
\begin{aligned}
 &\norm{\phi_*}_{W^{k,\infty}(-1,1)}
 +\norm{c_{+,*}-c_0}_{W^{k,\infty}(-1,1)}
 +\norm{c_{-,*}-c_0}_{W^{k,\infty}(-1,1)}
 \le C_k\delta_\zeta,\\
 &\norm{U_*-U_{\mathrm C}}_{W^{k,\infty}(-1,1)}
 \le C_k\frac{\gamma\lambda_{\mathrm D}^2}{\nu}
 |E_0|\delta_\zeta.
\end{aligned}
\end{equation}
\end{lemma}

\begin{proof}
Fix a nonnegative integer \(m\), set
\[
 X_m:=H^{m+2}(-1,1)\cap H_0^1(-1,1),
 \qquad
 Y_m:=H^m(-1,1),
\]
and write \(\phi_*=\ell_\zeta+w\).  Define
\[
 \mathcal F_m(w,\zeta_+,\zeta_-)
 :=-\lambda_{\mathrm D}^2w''
 -c_0\e^{-(\ell_\zeta+w)}
 +c_0\e^{\ell_\zeta+w}.
\]
Near the origin, \(\mathcal F_m:X_m\times\R^2\longrightarrow Y_m\) is \(C^1\), and
\[
 D_w\mathcal F_m(0,0,0)h
 =-\lambda_{\mathrm D}^2h''+2c_0h.
\]
For \(f\in Y_m\), the Dirichlet problem
\[
 -\lambda_{\mathrm D}^2h''+2c_0h=f,
 \qquad h(\pm1)=0,
\]
has a unique weak solution by the Lax--Milgram theorem.  One-dimensional elliptic
regularity upgrades it to \(h\in X_m\) and gives
\[
 \|h\|_{H^{m+2}}\le C_m\|f\|_{H^m}.
\]
It follows that \(D_w\mathcal F_m(0,0,0):X_m\longrightarrow Y_m\) is an isomorphism.
Hence the implicit function theorem gives \(\delta_m>0\) and a unique local branch
\(w=w(\zeta_+,\zeta_-)\in X_m\) satisfying
\[
 \|w\|_{H^{m+2}}\le C_m\delta_\zeta
 \qquad (\delta_\zeta\le\delta_m).
\]
Since \(\|\ell_\zeta\|_{H^{m+2}}\le C_m\delta_\zeta\), the corresponding branch
\(\phi=\ell_\zeta+w\) obeys the same bound in \(H^{m+2}\).  By the uniqueness established in Proposition~\ref{prop:exact}, this local branch is precisely the Poisson--Boltzmann solution \(\phi_*\) defined there.  Standard composition estimates for the smooth maps
\(z\mapsto c_0\e^{\mp z}\) then give
\[
 \|\phi_*\|_{H^{m+2}}
 +\|c_{+,*}-c_0\|_{H^{m+2}}
 +\|c_{-,*}-c_0\|_{H^{m+2}}
 \le C_m\delta_\zeta.
\]
Taking \(m=k\) and using the one-dimensional Sobolev embedding gives the first line of \eqref{eq:PB-small}.  The second follows from the identity for \(U_*\) in \eqref{eq:base-state}.
\end{proof}

\begin{lemma}[Graph-domain estimate]\label{lem:operator-difference}
For fixed physical parameters, there exist \(C_{\mathrm{PB}}>0\) and \(\delta_{\mathrm{op}}>0\) such that
\begin{equation}\label{eq:operator-difference}
 \norm{(\Lzeta-\Lzero)\bm Y}_{\Hphase}
 \le C_{\mathrm{PB}}\delta_\zeta\norm{\bm Y}_{\Dphase},
 \qquad
 \bm Y\in\Dphase,
 \qquad
 0\le\delta_\zeta\le\delta_{\mathrm{op}}.
\end{equation}
\end{lemma}

\begin{proof}
Fix \(\bm Y\in\Dphase\).  Recall that \(q_\zeta=b_+-b_-\) and \(\psi=\Kop q_\zeta\).  Lemma~\ref{lem:PB-small} with \(k=2\) gives
\begin{equation}\label{eq:coeff-small}
 \|U_*-U_{\mathrm C}\|_{W^{2,\infty}}
 +\|c_{+,*}-c_0\|_{W^{2,\infty}}
 +\|c_{-,*}-c_0\|_{W^{2,\infty}}
 +\|\phi_*\|_{W^{2,\infty}}
 +\|\varrho_*\|_{W^{2,\infty}}
 \le C\delta_\zeta.
\end{equation}
Dirichlet elliptic regularity applied to
\(
-\lambda_{\mathrm D}^2\Delta\psi=q_\zeta
\)
gives
\begin{equation}\label{eq:psi-H2}
 \|\psi\|_{H^2}
 \le C\lambda_{\mathrm D}^{-2}\|q_\zeta\|_{L^2}.
\end{equation}

After applying the Leray projector, the difference of the momentum components of
\(\Lzeta\bm Y\) and \(\Lzero\bm Y\) is the projection of
\[
 -(U_*-U_{\mathrm C})\partial_x\bm v
 -(U_*'-A)v_2\bm e_1
 -\gamma q_\zeta\nabla\phi_*
 -\gamma\varrho_*\nabla\psi.
\]
Since the Leray projector is bounded on \(L^2\),
\begin{align*}
 \|(U_*-U_{\mathrm C})\partial_x\bm v\|_{L^2}
 &\le C\delta_\zeta\|\bm v\|_{H^1},\\
 \|(U_*'-A)v_2\|_{L^2}
 &\le C\delta_\zeta\|\bm v\|_{L^2},\\
 \|q_\zeta\nabla\phi_*\|_{L^2}
 &\le C\delta_\zeta\|q_\zeta\|_{L^2},\\
 \|\varrho_*\nabla\psi\|_{L^2}
 &\le C\delta_\zeta\|\psi\|_{H^2}
 \le C\delta_\zeta\lambda_{\mathrm D}^{-2}\|q_\zeta\|_{L^2}.
\end{align*}
Hence the momentum component is bounded by
\(C\delta_\zeta\|\bm Y\|_{\Dphase}\).

For the positive species, subtracting the uniform coefficients gives
\[
 -(U_*-U_{\mathrm C})\partial_xb_+
 -v_2c_{+,*}'
 +D_+\Div(b_+\nabla\phi_*)
 +D_+\Div((c_{+,*}-c_0)\nabla\psi).
\]
The first two terms satisfy
\[
 \|(U_*-U_{\mathrm C})\partial_xb_+\|_{L^2}
 +\|v_2c_{+,*}'\|_{L^2}
 \le C\delta_\zeta
 \bigl(\|b_+\|_{H^1}+\|\bm v\|_{L^2}\bigr).
\]
Moreover,
\[
 \Div(b_+\nabla\phi_*)
 =\nabla b_+\cdot\nabla\phi_*+b_+\Delta\phi_*,
\]
which implies
\[
 \|\Div(b_+\nabla\phi_*)\|_{L^2}
 \le C\delta_\zeta\|b_+\|_{H^1}.
\]
Finally,
\[
 \Div((c_{+,*}-c_0)\nabla\psi)
 =\nabla(c_{+,*}-c_0)\cdot\nabla\psi
 +(c_{+,*}-c_0)\Delta\psi.
\]
Combining \eqref{eq:coeff-small}, \eqref{eq:psi-H2}, and
\(\Delta\psi=-q_\zeta/\lambda_{\mathrm D}^2\),  we obtain
\[
 \|\Div((c_{+,*}-c_0)\nabla\psi)\|_{L^2}
 \le C\delta_\zeta\lambda_{\mathrm D}^{-2}\|q_\zeta\|_{L^2}.
\]
For the negative species the coefficient difference is
\[
 -(U_*-U_{\mathrm C})\partial_x b_-
 -v_2c_{-,*}'
 -D_-\Div(b_-\nabla\phi_*)
 -D_-\Div((c_{-,*}-c_0)\nabla\psi).
\]
The first two terms are bounded by
\[
 C\delta_\zeta(\|b_-\|_{H^1}+\|\bm v\|_{L^2}),
\]
and the two divergence terms satisfy
\[
 \|\Div(b_-\nabla\phi_*)\|_{L^2}
 +\|\Div((c_{-,*}-c_0)\nabla\psi)\|_{L^2}
 \le C\delta_\zeta
 (\|b_-\|_{H^1}+\lambda_{\mathrm D}^{-2}\|q_\zeta\|_{L^2}).
\]
All physical parameters are fixed, so the factors involving \(D_\pm\), \(\gamma\), and \(\lambda_{\mathrm D}^{-1}\) are absorbed into \(C\).  Summing the fluid and the two ionic estimates proves \eqref{eq:operator-difference}.
\end{proof}

\begin{lemma}[Quadratic remainder at a weak layer]\label{lem:PB-nonlinear-map}
For \(\theta\in(3/4,1)\), the nonlinear remainder \(\Nzeta\) belongs to
\(C^1(\Xtheta,\Hphase)\) and satisfies
\[
 \|\Nzeta(\bm Y)\|_{\Hphase}\le C\|\bm Y\|_{\Xtheta}^2,
\]
\[
 \|\Nzeta(\bm Y)-\Nzeta(\widetilde{\bm Y})\|_{\Hphase}
 \le C\bigl(\|\bm Y\|_{\Xtheta}+\|\widetilde{\bm Y}\|_{\Xtheta}\bigr)
          \|\bm Y-\widetilde{\bm Y}\|_{\Xtheta}.
\]
For fixed physical parameters and \(\theta\), the constant \(C\) is independent of \(\zeta_\pm\).
\end{lemma}

\begin{proof}
Recall \(\bm Y=(\bm v,b_+,b_-)\), \(q_\zeta=b_+-b_-\), and
\(\psi=\Kop q_\zeta\).  After subtracting the exact steady equations and all
terms retained in the linearization \eqref{eq:PB-linearized}, the three
components of the nonlinear remainder are
\[
\begin{aligned}
 (\Nzeta(\bm Y))_{\mathrm{fl}}
 &=-\Pdiv(\bm v\cdot\nabla\bm v)
   -\gamma\Pdiv(q_\zeta\nabla\psi),\\
 (\Nzeta(\bm Y))_+
 &=-\bm v\cdot\nabla b_+
   +D_+\Div(b_+\nabla\psi),\\
 (\Nzeta(\bm Y))_-
 &=-\bm v\cdot\nabla b_-
   -D_-\Div(b_-\nabla\psi).
\end{aligned}
\]
Since \(\psi=\Kop q_\zeta\), this can be written as
\[
 \Nzeta(\bm Y)=
 \begin{pmatrix}
 -\Pdiv(\bm v\cdot\nabla\bm v)-\gamma\Pdiv(q_\zeta\nabla\Kop q_\zeta)\\[1mm]
 -\bm v\cdot\nabla b_+ +D_+\Div(b_+\nabla\Kop q_\zeta)\\[1mm]
 -\bm v\cdot\nabla b_- -D_-\Div(b_-\nabla\Kop q_\zeta)
 \end{pmatrix}.
\]
Under the identification
\[
 (a_+,a_-,q,\varphi)=(b_+,b_-,q_\zeta,\psi),
\]
the right-hand side is exactly the map \(\Nzero\) in Section~\ref{sec:nonlinear-proof}.  Hence
\[
 \Nzeta(\bm Y)=\Nzero(\bm Y).
\]
Lemma~\ref{lem:quadratic} therefore gives
\[
 \|\Nzeta(\bm Y)\|_{\Hphase}
 \le C\|\bm Y\|_{\Xtheta}^2
\]
and
\[
 \|\Nzeta(\bm Y)-\Nzeta(\widetilde{\bm Y})\|_{\Hphase}
 \le C\bigl(\|\bm Y\|_{\Xtheta}+\|\widetilde{\bm Y}\|_{\Xtheta}\bigr)
          \|\bm Y-\widetilde{\bm Y}\|_{\Xtheta}.
\]
Since \(\Nzeta\) is the same quadratic map in these coordinates, the constant
\(C\) has no additional dependence on \(\zeta_\pm\).
\end{proof}

Lemmas~\ref{lem:operator-difference} and~\ref{lem:PB-nonlinear-map} now provide the two inputs needed for Theorem~\ref{thm:PB-main}.

\begin{proof}[Proof of Theorem~\ref{thm:PB-main}]
By Theorem~\ref{thm:linear-main}, \(\Lzero\) generates an exponentially stable analytic semigroup on \(\Hphase\) with domain \(\Dphase\).  Lemma~\ref{lem:operator-difference} gives
\[
 \Lzeta=\Lzero+\mathcal B_\zeta,
 \qquad
 \|\mathcal B_\zeta\|_{\mathcal L(\Dphase,\Hphase)}\le C_{\mathrm{PB}}\delta_\zeta.
\]
Recall \(C_{\mathrm g}\ge1\) from Lemma~\ref{lem:graph-norm} and put \(c_{\mathrm g}=C_{\mathrm g}^{-1}\).  Then
\begin{equation}\label{eq:PB-graph-perturbation}
 \|\mathcal B_\zeta\bm Y\|_{\Hphase}
 \le \frac{C_{\mathrm{PB}}\delta_\zeta}{c_{\mathrm g}}
 \bigl(\|\bm Y\|_{\Hphase}+\|\Lzero\bm Y\|_{\Hphase}\bigr),
 \qquad \bm Y\in\Dphase.
\end{equation}
Let \(\omega_{\mathrm{lin}}=\omega_0/2\), and denote by
\(\varepsilon_{\mathrm{pert}}(\omega_{\mathrm{lin}})\) the threshold in Lemma~\ref{lem:stable-perturbation}.  If \(\delta_{\mathrm{prof}}\) and \(\delta_{\mathrm{op}}\) are the thresholds in Lemmas~\ref{lem:PB-small} and~\ref{lem:operator-difference}, choose \(\delta_*>0\) so that
\[
 \delta_*<\min\left\{
 \delta_{\mathrm{prof}},\delta_{\mathrm{op}},
 \frac{c_{\mathrm g}\varepsilon_{\mathrm{pert}}(\omega_{\mathrm{lin}})}{C_{\mathrm{PB}}},
 \frac{c_{\mathrm g}}{2C_{\mathrm{PB}}}
 \right\}.
\]
Lemma~\ref{lem:stable-perturbation} applied to \eqref{eq:PB-graph-perturbation} proves \eqref{eq:PB-main-linear}.  The same choice of \(\delta_*\) gives
\[
 \frac{c_{\mathrm g}}2\|\bm Y\|_{\Dphase}
 \le \|\bm Y\|_{\Hphase}+\|\Lzeta\bm Y\|_{\Hphase}
 \le C\|\bm Y\|_{\Dphase},
 \qquad \bm Y\in\Dphase,
\]
so
\[
 [\Hphase,\mathcal D(\Lzeta)]_\theta=\Xtheta
\]
with equivalent norms.

Let \(\omega_\zeta>0\) be the decay rate in \eqref{eq:PB-main-linear}.  Choose
\(0<\omega_{\zeta,\mathrm{sg}}<\omega_\zeta\) and \(C_{\zeta,\mathrm{sg}}\ge1\) so that
\[
\begin{aligned}
 \|\e^{t\Lzeta}\|_{\mathcal L(\Hphase,\Xtheta)}
 &\le C_{\zeta,\mathrm{sg}}t^{-\theta}\e^{-\omega_{\zeta,\mathrm{sg}}t},&&t>0,\\
 \|\e^{t\Lzeta}\|_{\mathcal L(\Xtheta)}
 &\le C_{\zeta,\mathrm{sg}}\e^{-\omega_{\zeta,\mathrm{sg}}t},&&t\ge0.
\end{aligned}
\]
Since
\[
 \partial_t\bm Y=\Lzeta\bm Y+\Nzeta(\bm Y),
\]
the fixed-point proof of Theorem~\ref{thm:nonlinear-main} applies with
\[
 \Lzero\mapsto\Lzeta,
 \qquad \Nzero\mapsto\Nzeta,
 \qquad \omega_1\mapsto\omega_{\zeta,\mathrm{sg}}.
\]
To make the smallness constants explicit, let \(C_{\mathrm Q}>0\) be the bilinear constant in Lemma~\ref{lem:PB-nonlinear-map}, put
\[
 \omega_{\zeta,\mathrm{nl}}=\frac12\omega_{\zeta,\mathrm{sg}},
 \qquad
 K_{\zeta,\theta}=C_{\zeta,\mathrm{sg}}
 \int_0^\infty r^{-\theta}
 \e^{-(\omega_{\zeta,\mathrm{sg}}-\omega_{\zeta,\mathrm{nl}})r}\dd r,
\]
and define
\[
 C_{\zeta,\mathrm{quad}}=K_{\zeta,\theta}C_{\mathrm Q},
 \qquad
 C_{\zeta,\mathrm{Lip}}=2K_{\zeta,\theta}C_{\mathrm Q},
 \qquad
 C_{\zeta,\mathrm{nl}}=2C_{\zeta,\mathrm{sg}}.
\]
These are the constants produced by the weighted Duhamel convolution and the quadratic difference estimate.  It is therefore enough to choose
\begin{equation}\label{eq:PB-initial-threshold}
 \varepsilon_\zeta<\min\left\{
 \frac{1}{4C_{\zeta,\mathrm{sg}}C_{\zeta,\mathrm{quad}}},
 \frac{1}{2C_{\zeta,\mathrm{sg}}C_{\zeta,\mathrm{Lip}}},
 \frac{m_*}{2C_{\mathrm{emb}}C_{\zeta,\mathrm{nl}}}
 \right\}.
\end{equation}
The contraction estimate then yields \eqref{eq:PB-main-nonlinear}.  The Volterra uniqueness estimate and the positive-time regularity argument from Theorem~\ref{thm:nonlinear-main}, with the same replacements above, give
\[
 \bm Y\in C((0,\infty);\Dphase)\cap C^1((0,\infty);\Hphase)
\]
and uniqueness in \(C([0,T];\Xtheta)\).  Since both \(\e^{t\Lzeta}\) and \(\Nzeta\) preserve real-valued states, the solution is real for real initial data.

The final term in \eqref{eq:PB-initial-threshold}, together with \eqref{eq:Xembed}, gives
\[
 c_{\pm,*}(y)+b_\pm(x,y,t)\ge\frac{m_*}{2},
 \qquad t\ge0.
\]
For \(t>0\), the pressure is recovered from
\[
\begin{aligned}
 \nabla\pi
 =(\Id-\Pdiv)\Bigl[
 &\nu\Delta\bm v-U_*\partial_x\bm v-v_2U_*'\bm e_1
 +\gamma q_\zeta(E_0\bm e_1-\nabla\phi_*)\\
 &-\gamma\varrho_*\nabla\psi
 -\bm v\cdot\nabla\bm v-\gamma q_\zeta\nabla\psi
 \Bigr].
\end{aligned}
\]
We choose the unique representative satisfying
\[
 \int_\Omega\pi(t)\dd x\dd y=0.
\]
Hence the projected solution satisfies the full perturbation system with homogeneous perturbation traces.
\end{proof}

\begin{remark}[A quantitative sufficient wall-potential condition]
The proof yields the explicit sufficient restriction
\[
 \delta_\zeta<\min\left\{
 \delta_{\mathrm{prof}},\delta_{\mathrm{op}},
 \frac{c_{\mathrm g}\varepsilon_{\mathrm{pert}}(\omega_0/2)}{C_{\mathrm{PB}}},
 \frac{c_{\mathrm g}}{2C_{\mathrm{PB}}}
 \right\}.
\]
Here \(c_{\mathrm g}=C_{\mathrm g}^{-1}\) converts the \(\Dphase\)-norm in Lemma~\ref{lem:operator-difference} to the graph norm used in Lemma~\ref{lem:stable-perturbation}.
\end{remark}

\begin{remark}[Blocking ionic walls]
Theorem~\ref{thm:PB-main} uses reservoir Dirichlet data, so the perturbation domain stays fixed as the wall potential varies.  For blocking walls, the no-flux boundary operator and the conserved ionic masses must instead be incorporated into the phase space; see \cite{ConstantinIgnatova2019,ConstantinIgnatovaLeeNear2022}.
\end{remark}

\appendix
\section{A perturbation lemma for analytic generators}\label{app:perturbation}

Section~\ref{sec:PB-proof} uses the following standard perturbation principle for analytic generators.  The analytic-generation and relative-perturbation ingredients are standard; see \cite[Chapter~2, Section~2.5 and Chapter~3, Section~3.2]{Pazy1983}.  For the final decay step we use the spectral mapping theorem for analytic semigroups, and hence the equality of the spectral and growth bounds; see \cite[Chapter~IV, Corollary~3.12]{EngelNagel2000}.  We include a short proof to record the graph-norm threshold used in Theorem~\ref{thm:PB-main}.

\begin{lemma}[Stable sectorial perturbations]\label{lem:stable-perturbation}
Let \(\mathscr X\) be a Banach space and let
\[
 \mathcal A:\mathcal D(\mathcal A)\subset\mathscr X\longrightarrow\mathscr X
\]
generate an analytic semigroup satisfying
\[
 \|\e^{t\mathcal A}\|_{\mathcal L(\mathscr X)}
 \le M\e^{-\omega_{\mathcal A}t},
 \qquad t\ge0,
\]
for some \(M\ge1\) and \(\omega_{\mathcal A}>0\).  Equip the domain with
\[
 \|x\|_{\mathcal D(\mathcal A)}
 :=\|x\|_{\mathscr X}+\|\mathcal Ax\|_{\mathscr X}.
\]
For \(0<\omega<\omega_{\mathcal A}\) there exists
\(\varepsilon_{\mathrm{pert}}(\omega)>0\) such that, if
\[
 \mathcal B\in\mathcal L(\mathcal D(\mathcal A),\mathscr X),
 \qquad
 \norm{\mathcal B}_{\mathcal L(\mathcal D(\mathcal A),\mathscr X)}
 <\varepsilon_{\mathrm{pert}}(\omega),
\]
then \(\mathcal A+\mathcal B\), with domain \(\mathcal D(\mathcal A)\),
generates an analytic semigroup whose decay rate can be chosen arbitrarily below
\(\omega\).
\end{lemma}

\begin{proof}
Fix \(0<\omega<\omega_{\mathcal A}\) and set
\[
 H_\omega=\{z\in\C:\operatorname{Re}z\ge-\omega\}.
\]
Since \(0<\omega<\omega_{\mathcal A}\), exponential stability gives \(H_\omega\subset\rho(\mathcal A)\).  On \(H_\omega\cap\{|z|\le R\}\), continuity of the resolvent and compactness give a uniform bound.  For sufficiently large \(|z|\) in \(H_\omega\), the standard resolvent estimate for analytic semigroups yields
\[
 \|(z-\mathcal A)^{-1}\|_{\mathcal L(\mathscr X)}
 \le \frac{C}{|z|}.
\]
Using
\[
 \mathcal A(z-\mathcal A)^{-1}
 =z(z-\mathcal A)^{-1}-\Id,
\]
we obtain the graph-resolvent bound
\[
 K_\omega:=\sup_{z\in H_\omega}
 \|(z-\mathcal A)^{-1}\|_{\mathcal L(\mathscr X,\mathcal D(\mathcal A))}<\infty.
\]
The standard relative perturbation theorem for analytic generators provides \(\varepsilon_{\mathrm{an}}>0\) such that graph-norm perturbations smaller than \(\varepsilon_{\mathrm{an}}\) preserve analyticity on \(\mathcal D(\mathcal A)\).  Choose
\[
 \varepsilon_{\mathrm{pert}}(\omega)
 <\min\left\{\varepsilon_{\mathrm{an}},\frac1{2K_\omega},\frac12\right\}.
\]
If \(\|\mathcal B\|_{\mathcal L(\mathcal D(\mathcal A),\mathscr X)}<\varepsilon_{\mathrm{pert}}(\omega)\), then
\[
 \|\mathcal B(z-\mathcal A)^{-1}\|_{\mathcal L(\mathscr X)}<\frac12,
 \qquad z\in H_\omega,
\]
so the Neumann formula
\[
 (z-\mathcal A-\mathcal B)^{-1}
 =(z-\mathcal A)^{-1}
 [\Id-\mathcal B(z-\mathcal A)^{-1}]^{-1}.
\]
shows that \(H_\omega\subset\rho(\mathcal A+\mathcal B)\).  In addition,
\[
 (1-\|\mathcal B\|)\|x\|_{\mathcal D(\mathcal A)}
 \le \|x\|_{\mathscr X}+\|(\mathcal A+\mathcal B)x\|_{\mathscr X}
 \le (1+\|\mathcal B\|)\|x\|_{\mathcal D(\mathcal A)},
\]
so \(\mathcal A+\mathcal B\) is closed on the same domain and generates an analytic semigroup.  Since
\(\sigma(\mathcal A+\mathcal B)\subset\{\operatorname{Re}z<-\omega\}\), the equality of the spectral and growth bounds for analytic semigroups \cite[Chapter~IV, Corollary~3.12]{EngelNagel2000} yields, for \(0<\omega'<\omega\),
\[
 \|\e^{t(\mathcal A+\mathcal B)}\|_{\mathcal L(\mathscr X)}
 \le C_{\omega'}\e^{-\omega't},
 \qquad t\ge0.
\]
\end{proof}

\end{document}